\documentclass[11pt, twoside, leqno]{article}
\usepackage{mathrsfs}
\usepackage{amssymb}
\usepackage{amsmath}
\usepackage{mathrsfs}
\usepackage{amsthm}
\usepackage{amsfonts}
\usepackage{color}
\usepackage{latexsym}
\usepackage{txfonts}
\usepackage{indentfirst}
\usepackage{soul}

\allowdisplaybreaks
\newtheorem{theorem}{Theorem}[section]
\newtheorem{lemma}[theorem]{Lemma}
\newtheorem{corollary}[theorem]{Corollary}
\newtheorem{proposition}[theorem]{Proposition}

\theoremstyle{definition}
\newtheorem{remark}[theorem]{Remark}
\newtheorem{definition}[theorem]{Definition}

\numberwithin{equation}{section}

\begin{document}

\title{\bf\Large Brezis--Seeger--Van Schaftingen--Yung-Type Characterization of
Homogeneous Ball Banach Sobolev Spaces and Its Applications
\footnotetext{\hspace{-0.35cm} 2020
{\it Mathematics Subject Classification}.
Primary 46E35; Secondary 35A23, 42B25, 26D10.\endgraf
{\it Key words and phrases.}
ball Banach function space, ball Banach Sobolev space,
Bourgain--Brezis--Mironescu formula,
Brezis--Seeger--Van Schaftingen--Yung formula,
$(\varepsilon,\infty)$-domain.
\endgraf
This project is supported by the National
Natural Science Foundation of China
(Grant Nos. 12122102, 11971058 and 12071197)
and the National Key Research
and Development Program of China
(Grant No. 2020YFA0712900).}}
\date{}
\author{Chenfeng Zhu,
Dachun Yang\footnote{Corresponding author,
E-mail: \texttt{dcyang@bnu.edu.cn}/{\color{red}
August 1, 2023}/Final version.}
\ and Wen Yuan}

\maketitle

\vspace{-0.7cm}

\begin{center}
\begin{minipage}{13cm}
{\small {\bf Abstract}\quad
Let $\gamma\in\mathbb{R}\setminus\{0\}$
and $X(\mathbb{R}^n)$ be a ball Banach function space
satisfying some extra mild assumptions.
Assume that $\Omega=\mathbb{R}^n$ or
$\Omega\subset\mathbb{R}^n$ is an $(\varepsilon,\infty)$-domain
for some $\varepsilon\in(0,1]$.
In this article, the authors prove that
a function $f$ belongs to the
homogeneous ball Banach Sobolev
space $\dot{W}^{1,X}(\Omega)$
if and only if $f\in L_{\mathrm{loc}}^1(\Omega)$ and
$$
\sup_{\lambda\in(0,\infty)}\lambda
\left\|\left[\int_{\{y\in\Omega:\
|f(\cdot)-f(y)|>\lambda|\cdot-y|^{1+\frac{\gamma}{p}}\}}
\left|\cdot-y\right|^{\gamma-n}\,dy
\right]^\frac{1}{p}\right\|_{X(\Omega)}<\infty,
$$
where $p\in[1,\infty)$ is related to $X(\mathbb{R}^n)$.
This result is of wide generality and can be applied
to various specific Sobolev-type function spaces,
including Morrey [Bourgain--Morrey-type,
weighted (or mixed-norm or variable) Lebesgue,
local (or global) generalized Herz,
Lorentz, and Orlicz (or Orlicz-slice)] Sobolev spaces,
which is new even in all these special cases; in particular,
it coincides with the well-known result
of H. Brezis, A. Seeger, J. Van Schaftingen, and P.-L. Yung
when $X(\Omega):=L^q(\mathbb{R}^n)$
with $1<p=q<\infty$, while
it is still new even when
$X(\Omega):=L^q(\mathbb{R}^n)$ with $1\leq p<q<\infty$.
The novelty of this article exists in that,
to establish the characterization of $\dot{W}^{1,X}(\Omega)$,
the authors provide
a machinery via using
the generalized Brezis--Seeger--Van Schaftingen--Yung
formula on $X(\mathbb{R}^n)$,
the extension theorem on $\dot{W}^{1,X}(\Omega)$,
the Bourgain--Brezis--Mironescu-type characterization of
the inhomogeneous ball Banach Sobolev
space $W^{1,X}(\Omega)$,
and the method of extrapolation
to overcome those difficulties caused by that $X(\mathbb{R}^n)$
might be neither the rotation invariance
nor the translation invariance and that
the norm of $X(\mathbb{R}^n)$ has no explicit expression.
}
\end{minipage}
\end{center}

\vspace{0.2cm}

\tableofcontents

\vspace{0.2cm}

\section{Introduction}
Let $\Omega\subset\mathbb{R}^n$ be an open set.
The \emph{homogeneous fractional Sobolev space}
$\dot{W}^{s,p}(\Omega)$ with
$s\in(0,1)$ and $p\in[1,\infty)$,
introduced by Gagliardo \cite{g1958},
plays an important role in studying many problems arising in
both harmonic analysis and partial
differential equations (see, for instance,
\cite{bbm2002,crs2010,cv2011,npv2012,h2008,ht2008,
ma2011,m2011}),
which is defined to be the set of
all the measurable functions $f$
on $\Omega$ having the
following finite \emph{Gagliardo semi-norm}
\begin{align}\label{Gnorm}
\|f\|_{\dot{W}^{s,p}(\Omega)}&:=
\left\|\frac{f(x)-f(y)}{|x-y|^{s+\frac{n}{p}}}
\right\|_{L^p(\Omega\times\Omega)}\\
&:=\left[\iint_{\Omega\times\Omega}
\frac{|f(x)-f(y)|^p}{|x-y|^{sp+n}}\,dx
\,dy\right]^{\frac{1}{p}}.\nonumber
\end{align}
Recall that the \emph{homogeneous Sobolev space}
$\dot{W}^{1,p}(\Omega)$ with $p\in[1,\infty)$
is defined to be the
set of all the locally integrable functions $f$ on $\Omega$
having the finite \emph{homogeneous Sobolev semi-norm}
$\|f\|_{\dot{W}^{1,p}(\Omega)}
:=\|\,|\nabla f|\,\|_{L^p(\Omega)}$,
where $\nabla f:=(\partial_1f,\ldots,\partial_nf)$
and, for any $j\in\{1,\ldots,n\}$, $\partial_j f$
is defined by setting,
for any $\phi\in C_{\mathrm{c}}^\infty(\Omega)$
(the set of all the infinitely differentiable
functions on $\Omega$ with compact support
in $\Omega$),
$$
\int_{\Omega}f(x)\partial_j\phi(x)\,dx
=-\int_{\Omega}\partial_jf(x)\phi(x)\,dx.
$$
It is well known that the Gagliardo semi-norm
$\|f\|_{\dot{W}^{s,p}(\Omega)}$
can not recover
the homogeneous Sobolev semi-norm
$\|\,|\nabla f|\,\|_{L^p(\Omega)}$
if we replace $s$ directly by $1$ in \eqref{Gnorm}
(see also \cite{bbm2000,bbm2002}).
In 2001, Bourgain et al. \cite{bbm2001}
gave a suitable way to recover
$\|\cdot\|_{\dot{W}^{1,p}(\Omega)}$
via $\|\cdot\|_{\dot{W}^{s,p}(\Omega)}$
in the case that $\Omega$ is a bounded smooth domain,
which was further extended by Brezis \cite{b2002}
to $\mathbb{R}^n$. To be more precise, Brezis \cite{b2002}
proved that,
for any given $p\in[1,\infty)$ and for any
$f\in\dot{W}^{1,p}(\mathbb{R}^n)\cap L^p(\mathbb{R}^n)$,
\begin{align}\label{1055}
\lim_{s\to1^-}(1-s)
\|f\|_{\dot{W}^{s,p}(\mathbb{R}^n)}^p
=\frac{2\pi^\frac{n-1}{2}
\Gamma(\frac{p+1}{2})}{p\Gamma(\frac{p+n}{2})}
\|\,|\nabla f|\,\|_{L^p(\mathbb{R}^n)}^p,
\end{align}
where $\Gamma$ is the Gamma function.
Here and thereafter,
$s\to1^-$ means $s\in(0,1)$ and $s\to1$.
The identity \eqref{1055} is nowadays referred to
as the Bourgain--Brezis--Mironescu formula on $\mathbb{R}^n$.
Recently,
Brezis et al. \cite{bvy2021}
discovered a striking novel way to
recover the
homogeneous Sobolev semi-norm
$\|\cdot\|_{\dot{W}^{1,p}(\mathbb{R}^n)}$
from the Gagliardo semi-norm $\|\cdot\|_{\dot{W}^{s,p}(\Omega)}$
via replacing the product $L^p$ norm
$\|\cdot\|_{L^p(\mathbb{R}^n\times\mathbb{R}^n)}$
in \eqref{Gnorm} by the weak
product $L^p$ quasi-norm $\|\cdot\|_{
L^{p,\infty}(\mathbb{R}^n\times\mathbb{R}^n)}$,
which was further extended
by Brezis et al. \cite{bsvy.arxiv}
from any $f\in C_{\mathrm{c}}^\infty(\mathbb{R}^n)$
to any $f\in\dot{W}^{1,p}(\mathbb{R}^n)$
and also from $\gamma=n$ to any $\gamma\in\mathbb{R}\setminus\{0\}$.
Properly speaking,
let $p=1$ and $\gamma\in\mathbb{R}\setminus[-1,0]$ or
let $p\in(1,\infty)$ and
$\gamma\in\mathbb{R}\setminus\{0\}$.
Brezis et al. \cite{bsvy.arxiv} proved that,
for any $f\in\dot{W}^{1,p}(\mathbb{R}^n)$,
\begin{align}\label{2118}
&\left\|\,\left|\nabla f\right|\,\right\|_{L^p(\mathbb{R}^n)}\\
&\quad\ \sim\left\|\frac{f(x)-f(y)}{|x-y|^{1+\frac{\gamma}{p}}}
\right\|_{L^{p,\infty}(\mathbb{R}^n\times
\mathbb{R}^n,\nu_\gamma)}\nonumber\\
&\quad:=\sup_{\lambda\in(0,\infty)}\lambda\left[\nu_\gamma
\left(\left\{(x,y)\in\mathbb{R}^n\times\mathbb{R}^n:\
\frac{|f(x)-f(y)|}{|x-y|^{1+\frac{\gamma}{p}}}
>\lambda\right\}\right)\right]^\frac{1}{p},\nonumber
\end{align}
where the implicit positive equivalence constants
are independent of $f$ and,
for any measurable set
$E\subset\mathbb{R}^n\times\mathbb{R}^n$,
\begin{align*}
\nu_\gamma(E):=\iint_{\{(x,y)\in E:\ x\neq y\}}
\left|x-y\right|^{\gamma-n}\,dx\,dy.
\end{align*}
This formula \eqref{2118} is referred to
as the Brezis--Seeger--Van Schaftingen--Yung formula on $\mathbb{R}^n$.
Both \eqref{1055} and \eqref{2118} have found
some  interesting applications.
Indeed,
Brezis et al. \cite{bvy2021} used \eqref{2118} with $\gamma:=n$
to obtain several
surprising weak-type estimates
of fractional Sobolev and fractional Gagliardo--Nirenberg
inequalities in some critical cases
involving $\dot{W}^{1,1}(\mathbb{R}^n)$,
where the anticipated fractional Sobolev
and the anticipated fractional
Gagliardo--Nirenberg estimates may fail.
In addition, using both \eqref{1055} on smooth bounded
domains and
\eqref{2118}, Brezis et al. \cite{bsvy.arxiv}
further established a
characterization of $\dot{W}^{1,p}(\mathbb{R}^n)$.
We refer the reader
to \cite{bsy2023,bn2016,
dm.arXiv,gh2023,hp.arXiv,ls2011,l2014,m2005,n2006}
for more studies related
to the Bourgain--Brezis--Mironescu formula \eqref{1055},
to \cite{bsvy,dlyyz2022,dlyyz.arxiv,
dm.arXiv1,dssvy,dt.arXiv,frank,gy2021}
for more studies related
to the Brezis--Seeger--Van Schaftingen--Yung formula \eqref{2118},
and to \cite{bm2018,bm2019,bvyCVPDE}
for more studies on the
Gagliardo--Nirenberg inequality.

Recall that the concept of
ball quasi-Banach function spaces,
introduced by Sawano et al. \cite{shyy2017},
unifies several important function spaces
in both partial differential equations
and harmonic analysis,
which is less restrictive than the classical
concept of quasi-Banach function spaces
introduced by Bennett and Sharpley \cite[Chapter 1]{bs1988}.
Indeed, there exist many examples of ball quasi-Banach function spaces,
such as Morrey spaces, weighted (or mixed-norm or variable) Lebesgue spaces,
Bourgain--Morrey-type spaces,
local (or global) generalized Herz spaces,
Lorentz spaces, and Orlicz (or Orlicz-slice) spaces.
However, some of the aforementioned
function spaces might not be the classical
quasi-Banach function spaces;
see, for instance, \cite[Section~7]{shyy2017}.
Function spaces based on ball quasi-Banach
function spaces have attracted a lot of attention and
made considerable progress in recent years;
see, for instance, \cite{cwyz2020,is2017,s2018,shyy2017,yyy2020}.
Very recently, Dai et al. \cite{dlyyz.arxiv}
subtly used the Poincar\'e inequality,
the exquisite geometry of $\mathbb{R}^n$,
the method of extrapolation, and the exact operator
norm of the Hardy--Littlewood maximal operator
on the associate space of the convexification of
the ball Banach function space $X(\mathbb{R}^n)$
to successfully extend \eqref{2118} with $\gamma=n$
from $\dot{W}^{1,p}(\mathbb{R}^n)$ to
the homogeneous ball Banach Sobolev space
$\dot{W}^{1,X}(\mathbb{R}^n)$
[see Definition~\ref{2.7}(i) for the definition of
$\dot{W}^{1,X}(\mathbb{R}^n)$]
and then further established some new fractional Sobolev and
some new fractional Gagliardo--Nirenberg inequalities
in various specific function spaces.
Motivated by the work in \cite{bsvy.arxiv}
and borrowing some
ideas and techniques from \cite{dlyyz.arxiv},
the results in \cite{dlyyz.arxiv} were further
generalized by Zhu et al. \cite{zyy2023}
from $\gamma=n$ to any $\gamma\in\mathbb{R}\setminus\{0\}$.
Meanwhile, Dai et al. \cite{dgpyyz2022} established
an analogue of \eqref{1055} for a given ball Banach function space
satisfying some extra mild hypotheses.
Later, Zhu et al. \cite{zyy2023bbm} also extended \eqref{1055}
to the setting of ball Banach function spaces
on bounded $(\varepsilon,\infty)$-domains with $\varepsilon\in(0,1]$.
For more studies on ball quasi-Banach function spaces,
we refer the reader
to \cite{h2021,wyy2020,zwyy2021} for the
boundedness of operators on ball quasi-Banach function spaces,
to \cite{syy1,syy2,wyy.arxiv,
yhyy2022-1,yhyy2022-2,zhyy2022} for Hardy spaces
associated with ball quasi-Banach function spaces,
and to \cite{hcy2021,ins2019,tyyz2021,wyyz2021}
for further applications of ball quasi-Banach function spaces.

Motivated by the aforementioned works,
let $\gamma\in\mathbb{R}\setminus\{0\}$
and $X(\mathbb{R}^n)$ be a ball Banach function space
satisfying some extra mild assumptions.
Assume that $\Omega=\mathbb{R}^n$ or
$\Omega\subset\mathbb{R}^n$ is an $(\varepsilon,\infty)$-domain
for some $\varepsilon\in(0,1]$.
In this article, we prove that
a function $f$ belongs to the
homogeneous ball Banach Sobolev
space $\dot{W}^{1,X}(\Omega)$
if and only if $f\in L_{\mathrm{loc}}^1(\Omega)$ and
\begin{align}\label{1628}
\sup_{\lambda\in(0,\infty)}\lambda
\left\|\left[\int_{\Omega}
\mathbf{1}_{E_{\lambda,\frac{\gamma}{p}}[f]}(\cdot,y)
\left|\cdot-y\right|^{\gamma-n}\,dy
\right]^\frac{1}{p}\right\|_{X(\Omega)}<\infty,
\end{align}
where $p\in[1,\infty)$ is related to $X(\mathbb{R}^n)$
and
\begin{align}\label{Elambda}
E_{\lambda,\frac{\gamma}{p}}[f]
:=\left\{(x,y)\in\Omega\times\Omega:\
x\neq y,\ \frac{|f(x)-f(y)|}{
|x-y|^{1+\frac{\gamma}{p}}}>\lambda\right\}
\end{align}
for any $\lambda\in(0,\infty)$;
moreover, if this holds true for a function $f$,
then
\begin{align}\label{1531}
\sup_{\lambda\in(0,\infty)}\lambda
\left\|\left[\int_{\Omega}
\mathbf{1}_{E_{\lambda,\frac{\gamma}{p}}[f]}(\cdot,y)
\left|\cdot-y\right|^{\gamma-n}\,dy
\right]^\frac{1}{p}\right\|_{X(\Omega)}
\sim\left\|\,|\nabla f|\,\right\|_{X(\Omega)},
\end{align}
where the positive equivalence constants are independent of $f$
(see Theorems~\ref{1931} and~\ref{1039} below).
This result is of wide generality and can be applied
to various specific Sobolev-type function spaces,
including Morrey spaces, Bourgain--Morrey-type spaces,
weighted (or mixed-norm or variable) Lebesgue spaces,
local (or global) generalized Herz spaces,
Lorentz spaces, and Orlicz (or Orlicz-slice) Sobolev spaces,
which is new even in all these special cases; in particular,
it coincides with the well-known result
of Brezis et al. \cite{bsvy.arxiv}
when $X(\Omega):=L^q(\mathbb{R}^n)$
with $1<p=q<\infty$, while
it is still new even when
$X(\Omega):=L^q(\mathbb{R}^n)$ with $1\leq p<q<\infty$.
The novelty of this article exists in that,
to establish the characterization of $\dot{W}^{1,X}(\Omega)$,
we provide a machinery via using
the generalized Brezis--Seeger--Van Schaftingen--Yung
formula on $X(\mathbb{R}^n)$,
the extension theorem on $\dot{W}^{1,X}(\Omega)$,
the Bourgain--Brezis--Mironescu-type characterization of
the inhomogeneous ball Banach Sobolev
space $W^{1,X}(\Omega)$,
and the method of extrapolation
to overcome those difficulties caused by that $X(\mathbb{R}^n)$
might be neither the rotation invariance
nor the translation invariance and that
the norm of $X(\mathbb{R}^n)$ has no explicit expression.

As two main results of this article,
Theorems~\ref{1931} and~\ref{1039}
present a characterization of $\dot{W}^{1,X}(\Omega)$
in terms of \eqref{1628};
moreover, \eqref{1531} gives an equivalence
between the homogeneous Sobolev semi-norm and the functional
only involving the difference without any derivatives.
It is remarkable that such a characterization holds true
in the setting of ball Banach function spaces.
As was pointed out in \cite{dlyyz.arxiv} that finding an
appropriate way to characterize the smoothness of functions
via their finite differences is indeed notoriously
difficult in approximation theory,
even for some weighted Lebesgue spaces
in the 1-dimensional case;
see \cite{k2015,mt2001} and the references therein.
A main difficulty comes from that the
difference operator
$\Delta_\tau(f):=f(\cdot+\tau)-f(\cdot)$
for any $\tau\in\mathbb{R}^n\setminus\{\mathbf{0}\}$
is no longer bounded on general weighted Lebesgue spaces.
It turns out in \cite{dlyyz.arxiv} that the method of extrapolation
(see Lemma~\ref{4.6} below) plays a key role
in getting rid of the strong dependence on
the explicit expression of the norm of the ball Banach function space
under consideration,
which is a bridge connecting
the weighted Lebesgue space and
the ball Banach function space.
Using this extrapolation method
and the exact operator norm
of the Hardy--Littlewood
maximal operator
on the associate space
of the ball Banach function space, the estimate \eqref{1531} can be
deduced from a similar estimate to the one in weighted Lebesgue
spaces with Muckenhoupt $A_1(\mathbb{R}^n)$ weights.
On the other hand, we point out that the proofs
of Theorems~\ref{1931} and~\ref{1039} are far from trivial,
which particularly need to use two very recently developed new tools,
the generalized Brezis--Seeger--Van Schaftingen--Yung
formula on $X(\mathbb{R}^n)$
(see Lemma~\ref{1016} below)
and the Bourgain--Brezis--Mironescu-type characterization of
the inhomogeneous ball Banach Sobolev
space $W^{1,X}(\Omega)$ (see Lemma~\ref{2006} below).
Indeed, these two useful tools play an important role in the proof
of the main estimates of this article, which are used to
overcome the difficulties caused by the deficiency of
the rotation invariance and the translation invariance
of ball Banach function spaces.

The organization of the remainder of this article is as follows.

In Section~\ref{section2},
we first give some preliminaries on ball quasi-Banach function
spaces and then recall the concepts of both the
(in)homogeneous ball Banach Sobolev space
and the Muckenhoupt class.

Section~\ref{section5} is devoted to establishing
the Brezis--Seeger--Van Schaftingen--Yung-type
characterization of $\dot{W}^{1,X}(\Omega)$
in terms of \eqref{1628}.
In Subsection~\ref{sub3.1},
using the Brezis--Seeger--Van Schaftingen--Yung formula
on the ball Banach function space $X(\mathbb{R}^n)$
(see Lemma~\ref{1016} below),
the Bourgain--Brezis--Mironescu-type characterization
of $W^{1,X}(\Omega)$
(see Lemma~\ref{2006} below),
the method of extrapolation,
and the Alaoglu theorem, we obtain a characterization of
$\dot{W}^{1,X}(\mathbb{R}^n)$
under some extra mild assumptions on $X(\mathbb{R}^n)$
(see Theorem~\ref{1931} below).
In Subsection~\ref{sub3.2},
we establish a characterization of $\dot{W}^{1,X}(\Omega)$
(see Theorem~\ref{1039} below)
via using the extension theorem on $\dot{W}^{1,X}(\Omega)$
and the aforementioned tools used in the proof
of Theorem~\ref{1931},
where $\Omega\subset\mathbb{R}^n$
is an $(\varepsilon,\infty)$-domain
for some $\varepsilon\in(0,1]$.

In Section~\ref{S5},
we apply Theorems~\ref{1931} and~\ref{1039}
to various specific ball Banach function spaces, such as
weighted Lebesgue spaces
and Morrey spaces (see Subsection~\ref{5.4} below),
Besov--Bourgain--Morrey spaces
(see Subsection~\ref{BBMspace} below),
local generalized Herz spaces
and global generalized Herz spaces
(see Subsection~\ref{Herz} below),
mixed-norm Lebesgue spaces (see Subsection~\ref{5.2} below),
variable Lebesgue spaces (see Subsection~\ref{5.3} below),
Lorentz spaces (see Subsection~\ref{5.7} below),
Orlicz spaces (see Subsection~\ref{5.5} below),
and Orlicz-slice spaces (see Subsection~\ref{5.6} below).
All of these results are new.
Obviously, more applications of the results of this article
(for instance, to some newfound function spaces)
are predictable
due to their flexibility and their generality.

At the end of this section, we make some conventions on notation.
We always let $\mathbb{N}:=\{1,2,\ldots\}$
and $\mathbb{S}^{n-1}$
be the unit sphere of $\mathbb{R}^n$.
We denote by $C$ a \emph{positive constant} which is independent
of the main parameters involved, but may vary from line to line.
We use $C_{(\alpha,\dots)}$ to denote a positive constant depending
on the indicated parameters $\alpha,\, \dots$.
The symbol $f\lesssim g$ means $f\le Cg$
and, if $f\lesssim g\lesssim f$, we then write $f\sim g$.
If $f\le Cg$ and $g=h$ or $g\le h$,
we then write $f\lesssim g=h$ or $f\lesssim g\le h$.
For any index $q\in[1,\infty]$,
its \emph{conjugate index} is denoted by $q'$,
that is, $\frac{1}{q}+\frac{1}{q'}=1$.
For any measurable
set $E\subset\mathbb{R}^n$,
we denote by ${\mathbf{1}}_E$ its
\emph{characteristic function},
by $E^\complement$ the set $\mathbb{R}^n\setminus E$,
and by $|E|$ its $n$-dimensional Lebesgue measure.
Moreover, we denote by $\sigma_{n-1}$
the $(n-1)$-dimensional Lebesgue measure of the unit sphere
$\mathbb{S}^{n-1}$ of $\mathbb{R}^n$.
For any $x\in\mathbb{R}^n$ and $r\in(0,\infty)$,
let $B(x,r):=\{y\in\mathbb{R}^n:\ |x-y|<r\}$.
Usually, for any $r\in(0,\infty)$,
we let $B_r:=B({\bf 0},r)$,
where $\mathbf{0}$ denotes the \emph{origin} of $\mathbb{R}^n$.
Let $\Omega\subset\mathbb{R}^n$ be an open set.
For any $p\in(0,\infty]$,
the \emph{Lebesgue space}
$L^p(\Omega)$
is defined to be the set of
all the measurable functions $f$
on $\Omega$ such that
\begin{align*}
\|f\|_{L^p(\Omega)}:=
\left[\int_{\Omega}|f(x)|^p\,dx\right]^\frac{1}{p}<\infty.
\end{align*}
For any $p\in(0,\infty)$,
the set of all the locally $p$-integrable
functions on $\Omega$,
$L_{\mathrm{loc}}^p(\Omega)$,
is defined to be the set of all the
measurable functions $f$
on $\Omega$ such that,
for any open set $V\subset\Omega$
satisfying that $V$ is compactly contained in $\Omega$,
$f\in L^p(V)$.
For any $f\in L^1_{\mathrm{loc}}(\mathbb{R}^n)$,
its \emph{Hardy--Littlewood maximal function} $\mathcal{M}(f)$
is defined by setting, for any $x\in\mathbb{R}^n$,
\begin{align*}
\mathcal{M}(f)(x):=\sup_{B\ni x}
\frac{1}{|B|}\int_B\left|f(y)\right|\,dy,
\end{align*}
where the supremum is taken over all the
balls $B\subset\mathbb{R}^n$ containing $x$.
Finally, when we prove a theorem (or the like),
in its proof we always use the same symbols as in the
statement itself of that theorem (or the like).

\section{Ball Banach Function (Sobolev) Spaces}
\label{section2}

In this section, we recall the concepts of
the ball quasi-Banach function space
as well as its convexification and its associate space,
the (in)homogeneous ball Banach Sobolev space,
and the Muckenhoupt class $A_p(\mathbb{R}^n)$.
First, we give some preliminaries on ball quasi-Banach function
spaces introduced in \cite[Definition~2.2]{shyy2017}.
Throughout this article,
the \emph{symbol} $\mathscr{M}(\mathbb{R}^n)$ denotes
the set of all measurable functions
on $\mathbb{R}^n$.

\begin{definition}\label{1659}
A quasi-Banach space $X(\mathbb{R}^n)
\subset\mathscr{M}(\mathbb{R}^n)$,
equipped with a quasi-norm $\|\cdot\|_{X(\mathbb{R}^n)}$
which makes sense for all functions
in $\mathscr{M}(\mathbb{R}^n)$,
is called a \emph{ball quasi-Banach function space} if
it has the following properties:
\begin{enumerate}
\item[\textup{(i)}]
for any $f\in\mathscr{M}(\mathbb{R}^n)$,
if $\|f\|_{X(\mathbb{R}^n)}=0$, then $f=0$ almost everywhere
in $\mathbb{R}^n$;
\item[\textup{(ii)}]
if $f,g\in\mathscr{M}(\mathbb{R}^n)$
with $|g|\leq|f|$ almost everywhere
in $\mathbb{R}^n$,
then $\|g\|_{X(\mathbb{R}^n)}
\leq\|f\|_{X(\mathbb{R}^n)}$;
\item[\textup{(iii)}]
if a sequence $\{f_m\}_{m\in\mathbb{N}}
\subset\mathscr{M}(\mathbb{R}^n)$
and $f\in\mathscr{M}(\mathbb{R}^n)$ satisfy
that $0\leq f_m\uparrow f$ almost everywhere in $\mathbb{R}^n$
as $m\to\infty$, then $\|f_m\|_{X(\mathbb{R}^n)}
\uparrow\|f\|_{X(\mathbb{R}^n)}$ as $m\to\infty$;
\item[\textup{(iv)}]
for any ball $B\subset\mathbb{R}^n$,
$\mathbf{1}_B\in X(\mathbb{R}^n)$.
\end{enumerate}
Moreover, a ball quasi-Banach function
space $X(\mathbb{R}^n)$ is called a
\emph{ball Banach function space} if
\begin{enumerate}
\item[\textup{(v)}]
for any $f,g\in X(\mathbb{R}^n)$,
$$
\|f+g\|_{X(\mathbb{R}^n)}\leq
\|f\|_{X(\mathbb{R}^n)}+\|g\|_{X(\mathbb{R}^n)};
$$
\item[\textup{(vi)}]
for any ball $B\subset\mathbb{R}^n$,
there exists a positive constant $C_{(B)}$
such that, for any $f\in X(\mathbb{R}^n)$,
$$
\int_B\left|f(x)\right|\,dx\leq
C_{(B)}\|f\|_{X(\mathbb{R}^n)}.
$$
\end{enumerate}
\end{definition}

\begin{remark}\label{1052}
\begin{enumerate}
\item[\textup{(i)}]
Let $X(\mathbb{R}^n)$ be a ball quasi-Banach function space.
By \cite[Remark~2.5(i)]{yhyy2022-1},
we find that, for any $f\in\mathscr{M}(\mathbb{R}^n)$,
$\|f\|_{X(\mathbb{R}^n)}=0$ if and only if $f=0$ almost everywhere
in $\mathbb{R}^n$.
\item[\textup{(ii)}]
As was mentioned in \cite[Remark~2.5(ii)]{yhyy2022-1},
we obtain an equivalent formulation
of Definition~\ref{1659} via replacing any
ball $B$ therein by any bounded measurable set $E$.
\item[\textup{(iii)}]
In Definition~\ref{1659}, if we replace any ball $B$ by any
measurable set $E$ with $|E|<\infty$,
then we obtain the definition of
(quasi-)Banach function spaces, which was originally introduced
by Bennett and Sharpley in
\cite[Chapter 1,
Definitions~1.1 and~1.3]{bs1988}.
Thus, a (quasi-)Banach function space is always a ball
(quasi-)Banach function space, but the converse is
not necessary to be true.
\item[\textup{(iv)}]
By \cite[Proposition~1.2.36]{lyh2320},
we find that both (ii) and (iii) of
Definition~\ref{1659} imply that any ball quasi-Banach function
space is complete.
\item[\rm(v)]
Let $X(\mathbb{R}^n)$ be a ball quasi-Banach function space.
Then, from both (ii) and (iv) of Definition~\ref{1659},
we infer that $C_{\mathrm{c}}^\infty
(\mathbb{R}^n)\subset X(\mathbb{R}^n)$.
\end{enumerate}
\end{remark}

Now, we recall the
concept of the $p$-convexification of
quasi-normed vector spaces;
see, for instance, \cite[Definition~2.6]{shyy2017}.

\begin{definition}\label{tuhua}
Let $p\in(0,\infty)$.
The \emph{$p$-convexification} $X^p(\mathbb{R}^n)$ of
the quasi-normed vector space
$X(\mathbb{R}^n)$ is defined by setting
$X^p(\mathbb{R}^n):=\{f\in\mathscr{M}(\mathbb{R}^n):\
|f|^p\in X(\mathbb{R}^n)\}$
equipped with the quasi-norm
$\|f\|_{X^p(\mathbb{R}^n)}:=\|\,|f|^p\|_{X(\mathbb{R}^n)}^\frac{1}{p}$
for any $f\in X^p(\mathbb{R}^n)$.
\end{definition}

The following concept
can be found in \cite[Section~2]{lz2023};
see \cite{lz20231,z67} for more details.

\begin{definition}\label{016}
A vector space $X(\mathbb{R}^n)\subset\mathscr{M}(\mathbb{R}^n)$
equipped with a quasi-norm $\|\cdot\|_{X(\mathbb{R}^n)}$
is called a \emph{quasi-Banach function space} over $\mathbb{R}^n$
if it has the following properties:
\begin{enumerate}
\item[\rm(i)]
for any $f\in X(\mathbb{R}^n)$,
if $\|f\|_{X(\mathbb{R}^n)}=0$, then $f=0$ almost everywhere
in $\mathbb{R}^n$;
\item[\rm(ii)]
if $f\in X(\mathbb{R}^n)$ and $g\in\mathscr{M}(\mathbb{R}^n)$
with $|g|\leq|f|$ almost everywhere in $\mathbb{R}^n$,
then $g\in X(\mathbb{R}^n)$
and $\|g\|_{X(\mathbb{R}^n)}\leq\|f\|_{X(\mathbb{R}^n)}$;
\item[\rm(iii)]
if a sequence $\{f_m\}_{m\in\mathbb{N}}
\subset X(\mathbb{R}^n)$
and $f\in\mathscr{M}(\mathbb{R}^n)$ satisfy
that $0\leq f_m\uparrow f$ almost everywhere in $\mathbb{R}^n$
as $m\to\infty$
and $\sup_{m\in\mathbb{N}}\|f_m\|_{X(\mathbb{R}^n)}<\infty$,
then $f\in X(\mathbb{R}^n)$ and $\|f\|_{X(\mathbb{R}^n)}
=\sup_{m\in\mathbb{N}}\|f_m\|_{X(\mathbb{R}^n)}$;
\item[\rm(iv)]
for any measurable set $E\subset\mathbb{R}^n$
of positive measure, there exists a measurable set $F\subset E$
of positive measure satisfying that $\mathbf{1}_F\in X(\mathbb{R}^n)$.
\end{enumerate}
\end{definition}

Definition \ref{016}(iv) is called the \emph{saturation property}.
The following proposition shows that,
if the quasi-normed vector space $X(\mathbb{R}^n)$ satisfies that
the Hardy--Littlewood maximal operator is weakly bounded on
its convexification,
then Definition~\ref{016} coincides with Definition~\ref{1659};
this proposition can be deduced from \cite[Proposition~4.22]{narxiv} and
we omit the details.

\begin{proposition}\label{839}
Let $X(\mathbb{R}^n)\subset\mathscr{M}(\mathbb{R}^n)$
be a vector space equipped with a quasi-norm $\|\cdot\|_{X(\mathbb{R}^n)}$.
Assume that there exists $r\in(0,\infty)$ such that
the Hardy--Littlewood maximal operator $\mathcal{M}$
is weakly bounded on $X^r(\mathbb{R}^n)$, that is,
there exists a positive constant $C$ such that,
for any $f\in X^r(\mathbb{R}^n)$,
$$
\sup_{\lambda\in(0,\infty)}\lambda
\left\|\mathbf{1}_{\{x\in\mathbb{R}^n:\
\mathcal{M}(f)(x)>\lambda\}}\right\|_{X^r(\mathbb{R}^n)}
\leq C\|f\|_{X^r(\mathbb{R}^n)}.
$$
Then $X(\mathbb{R}^n)$ is a quasi-Banach function space
in the sense of Definition~\ref{016} if and only if
$X(\mathbb{R}^n)$ is a ball quasi-Banach function space
in the sense of Definition~\ref{1659}.
\end{proposition}

\begin{remark}\label{920}
\begin{enumerate}
\item[\rm(i)] Let $X(\mathbb{R}^n)$ be a quasi-Banach function space
in the sense of Definition~\ref{016}. Nieraeth \cite[Proposition~4.22]{narxiv}
showed that, if the Hardy--Littlewood maximal operator $\mathcal{M}$
is weakly bounded on $X(\mathbb{R}^n)$,
then $\mathbf{1}_B\in X(\mathbb{R}^n)$ for any ball $B\subset\mathbb{R}^n$.
Obviously, this assumption is slightly stronger than the corresponding assumption
in Proposition~\ref{839}. Indeed, the corresponding assumption
in Proposition~\ref{839} and \cite[Proposition~4.22]{narxiv} imply that,
for any ball $B\subset\mathbb{R}^n$,
$\mathbf{1}_B\in X^r(\mathbb{R}^n)$
which is equivalent to $\mathbf{1}_B\in X(\mathbb{R}^n)$.
\item[\rm(ii)]
By Proposition~\ref{839} and \cite[Remark~2.6]{narxiv},
we conclude that any ball (quasi-)Banach function space
is always a (quasi-)Banach function space in the sense
of Definition~\ref{016};
on the other hand, any (quasi-)Banach function space
satisfying the assumption of Proposition~\ref{839}
is also a ball (quasi-)Banach function space.
A similar fact has already been observed in the remarks after
\cite[Theorem 6.2]{lz2023}.
\item[\rm(iii)]
From (i) of this remark, we further infer that,
under the extra assumption that the Hardy--Littlewood
maximal operator is weakly bounded on some convexification of
$X(\mathbb{R}^n)$, working with ball (quasi-)Banach function
spaces in the sense of Definition~\ref{1659}
or working with (quasi-)Banach function spaces in
the sense of Definition~\ref{016} would yield
exactly the same results.
\end{enumerate}
\end{remark}

In the remainder of this section,
let $\Omega$ be an open subset of $\mathbb{R}^n$.
The following concept of the restrictive
space of a ball quasi-Banach function space
can be found in \cite[Definition~2.6]{zyy2023bbm}.
Let $\mathscr{M}(\Omega)$ denote the set of
all measurable functions on $\Omega$.
For any $g\in\mathscr{M}(\mathbb{R}^n)$,
denote by $g|_\Omega$ the \emph{restriction} of $g$ on $\Omega$.

\begin{definition}\label{1635}
Let $X(\mathbb{R}^n)$ be a ball quasi-Banach function space.
The \emph{restrictive space} $X(\Omega)$ of
$X(\mathbb{R}^n)$ on $\Omega$ is defined by setting
\begin{align*}
X(\Omega):=\left\{f\in\mathscr{M}(\Omega):\
f=g|_\Omega\text{ for some }g\in X(\mathbb{R}^n)\right\};
\end{align*}
moreover, for any $f\in X(\Omega)$, let
$\|f\|_{X(\Omega)}:=\inf\{\|g\|_{X(\mathbb{R}^n)}:\
f=g|_\Omega,\,g\in X(\mathbb{R}^n)\}$.
\end{definition}

\begin{remark}\label{norm}
Let $X(\mathbb{R}^n)$ be a ball quasi-Banach function space
and $X(\Omega)$ its restrictive space.
\begin{enumerate}
\item[\rm(i)]
By \cite[Proposition~2.7]{zyy2023bbm},
we find that, for any $f\in X(\Omega)$,
$\|f\|_{X(\Omega)}=
\|\widetilde{f}\|_{X(\mathbb{R}^n)}$,
where
\begin{align}\label{1448}
\widetilde{f}(x):=
\begin{cases}
f(x)&\textup{ if }x\in\Omega\\
0&\textup{ if }x\in\Omega^\complement.
\end{cases}
\end{align}
\item[\rm(ii)]
From \cite[Proposition~2.8]{zyy2023bbm},
we deduce that $X(\Omega)$ satisfies all
the conditions in Definition~\ref{1659}
with $\mathbb{R}^n$ replaced by $\Omega$.
\item[\rm(iii)]
By both (ii) and (iv) of
\cite[Proposition~2.8]{zyy2023bbm},
we find that $C_{\mathrm{c}}^\infty(\Omega)
\subset X(\Omega)$.
\end{enumerate}
\end{remark}

We recall the concept of a ball quasi-Banach function space
with absolutely continuous quasi-norm, which
can be found
in \cite[Chapter 1, Definition~3.1]{bs1988};
see also \cite[Definition~3.2]{wyy2020}.

\begin{definition}\label{2029}
Let $X(\mathbb{R}^n)$ be a ball quasi-Banach function space
and $X(\Omega)$ its restrictive space.
Then $X(\Omega)$ is said
to have an \emph{absolutely continuous quasi-norm on $\Omega$}
if, for any $f\in X(\Omega)$ and any sequence $\{E_j\}_{j\in\mathbb{N}}$
of measurable subsets of $\Omega$
satisfying that $\mathbf{1}_{E_j}\to0$
almost everywhere in $\Omega$ as $j\to\infty$, one has
$\|f\mathbf{1}_{E_j}\|_{X(\Omega)}\to0$ as $j\to\infty$.
\end{definition}

Next, we recall the
concept of the associate space of a ball
Banach function space, which can be found in
\cite[Definition~2.10]{zyy2023bbm};
see \cite[p.\,9]{shyy2017} and
\cite[Chapter 1, Section 2]{bs1988} for more details.

\begin{definition}\label{associte}
Let $X(\mathbb{R}^n)$ be a ball Banach function space
and $X(\Omega)$ its restrictive space.
The \emph{associate space}
(also called the \emph{K\"othe dual}) $X'(\Omega)$
of $X(\Omega)$
is defined by setting
\begin{align*}
X'(\Omega):=\left\{f\in\mathscr{M}(\Omega):\
\|f\|_{X'(\Omega)}:=\sup_{\{g\in X(\Omega):\
\|g\|_{X(\Omega)}=1\}}
\left\|fg\right\|_{L^1(\Omega)}<\infty\right\},
\end{align*}
where $\|\cdot\|_{X'(\Omega)}$ is called the
\emph{associate norm} of $\|\cdot\|_{X(\Omega)}$.
\end{definition}

\begin{remark}\label{dual}
Let $X(\mathbb{R}^n)$ be a ball Banach function space
and $X(\Omega)$ its restrictive space.
\begin{enumerate}
\item[(i)]
From \cite[Proposition~2.3]{shyy2017},
we infer that $X'(\mathbb{R}^n)$ is also a ball Banach function space.
\item[(ii)]
By \cite[Lemma~2.6]{zwyy2021},
we find that $X(\Omega)$ coincides with its second
associate space $X''(\Omega)$.
\item[(iii)]
If $X(\Omega)$ has an absolutely continuous norm on $\Omega$,
then, from an argument similar to that
used in the proof of \cite[Chapter 1, Corollary~4.3]{bs1988},
we deduce that $X'(\Omega)$ coincides with $X^*(\Omega)$.
Here and thereafter,
$X^*(\Omega)$ denotes the \emph{dual space} of $X(\Omega)$.
\item[(iv)]
The associate space of $X(\Omega)$ in Definition~\ref{associte}
coincides with the set
of restrictions of $X'(\mathbb{R}^n)$ functions on $\Omega$
with the same norms; see \cite[Proposition 2.12]{zyy2023bbm}.
\end{enumerate}
\end{remark}

Now, we recall the concept of ball Banach Sobolev spaces;
see, for instance, \cite[Definition~2.4]{dlyyz.arxiv},
\cite[Definition~2.6]{dgpyyz2022},
and \cite[Definition~2.14]{zyy2023bbm}.

\begin{definition}\label{2.7}
Let $X(\mathbb{R}^n)$ be a ball Banach function space
and $X(\Omega)$ its restrictive space.
\begin{enumerate}
\item[\rm(i)]
The \emph{homogeneous ball Banach
Sobolev space} $\dot{W}^{1,X}(\Omega)$
is defined to be the set of all the
$f\in L_{\mathrm{loc}}^1(\Omega)$
such that $|\nabla f|\in X(\Omega)$ equipped
with the semi-norm
$$
\|f\|_{\dot{W}^{1,X}(\Omega)}:=
\left\|\,|\nabla f|\,\right\|_{X(\Omega)}.
$$
\item[\rm(ii)]
The \emph{inhomogeneous ball Banach
Sobolev space} $W^{1,X}(\Omega)$
is defined to be the set of all the
$f\in X(\Omega)$ such that $|\nabla f|\in X(\Omega)$ equipped
with the norm
$$
\|f\|_{W^{1,X}(\Omega)}:=\|f\|_{X(\Omega)}+
\left\|\,|\nabla f|\,\right\|_{X(\Omega)}.
$$
\end{enumerate}
\end{definition}

Next, we recall the concept of the
Muckenhoupt $A_p(\mathbb{R}^n)$ class
(see, for instance, \cite[Definitions~7.1.1 and~7.1.3]{g2014}).

\begin{definition}\label{1557}
An \emph{$A_p(\mathbb{R}^n)$-weight} $\omega$, with $p\in[1,\infty)$,
is a nonnegative locally integrable function
on $\mathbb{R}^n$ satisfying that,
when $p=1$,
\begin{align}\label{A1}
[\omega]_{A_1(\mathbb{R}^n)}:=\sup_{Q\subset\mathbb{R}^n}
\frac{\|\omega^{-1}\|_{L^\infty(Q)}}{|Q|}\int_Q\omega(x)\,dx<\infty
\end{align}
and, when $p\in(1,\infty)$,
\begin{align}\label{Ap}
[\omega]_{A_p(\mathbb{R}^n)}:=\sup_{Q\subset\mathbb{R}^n}
\frac{1}{|Q|}\int_Q\omega(x)\,dx
\left\{\frac{1}{|Q|}\int_Q
[\omega(x)]^{1-p'}\,dx\right\}^{p-1}<\infty,
\end{align}
where $p'$ denotes the conjugate index of $p$ and
the suprema in both \eqref{A1} and \eqref{Ap}
are taken over all cubes $Q\subset\mathbb{R}^n$.
Moreover, let
$A_\infty(\mathbb{R}^n):=\bigcup_{p\in[1,\infty)}A_p(\mathbb{R}^n).$
\end{definition}

Recall that the \emph{weighted Lebesgue space}
$L^r_\omega(\Omega)$,
with both $r\in(0,\infty)$
and $\omega\in A_\infty(\mathbb{R}^n)$,
is defined to be the set of all the
$f\in\mathscr{M}(\Omega)$ such that
\begin{align}\label{2040}
\|f\|_{L^r_\omega(\Omega)}:=
\left[\int_{\Omega}\left|f(x)\right|^r
\omega(x)\,dx\right]^{\frac{1}{r}}<\infty
\end{align}
with the usual modification made when $r=\infty$.

The following
basic properties and conclusions of
Muckenhoupt $A_p(\mathbb{R}^n)$-weights
can be found in
\cite[Proposition~7.1.5 and Theorem~7.1.9]{g2014}
and \cite[Theorem~2.7.4]{dhhr2011}.

\begin{lemma}\label{ApProperty}
Let $p\in[1,\infty)$ and $\omega\in A_p(\mathbb{R}^n)$.
Then the following statements hold true:
\begin{enumerate}
\item[\textup{(i)}]
$\omega\in A_q(\mathbb{R}^n)$ for any $q\in[p,\infty)$;
moreover, $[\omega]_{A_q(\mathbb{R}^n)}
\leq[\omega]_{A_p(\mathbb{R}^n)}$;
\item[\textup{(ii)}]
if $p\in(1,\infty)$,
then $\omega^{1-p'}\in A_{p'}(\mathbb{R}^n)$ and
$[L_\omega^p(\mathbb{R}^n)]'=
L_{\omega^{1-p'}}^{p'}(\mathbb{R}^n)$;
\item[\textup{(iii)}]
if $p\in(1,\infty)$, then the Hardy--Littlewood
maximal operator $\mathcal{M}$ is bounded
on $L^p_\omega(\mathbb{R}^n)$;
moreover, there exists a constant $C\in(0,\infty)$,
independent of $\omega$, such that
$$
\left\|\mathcal{M}\right\|_{L^p_\omega(\mathbb{R}^n)
\to L^p_\omega(\mathbb{R}^n)}
\leq C[\omega]_{A_p(\mathbb{R}^n)}^{p'-1},
$$
where $\|\mathcal{M}\|_{L^p_\omega
(\mathbb{R}^n)\to L^p_\omega(\mathbb{R}^n)}$
denotes the operator norm of
$\mathcal{M}$ on $L^p_\omega(\mathbb{R}^n)$.
\end{enumerate}
\end{lemma}

\section{Characterization of Homogeneous Ball
Banach Sobolev Spaces}
\label{section5}

Let $X(\mathbb{R}^n)$ be a ball Banach function space,
$\Omega\subset\mathbb{R}^n$ an open set,
and $\dot{W}^{1,X}(\Omega)$ the homogeneous ball
Banach Sobolev space.
The main target of this section is to establish
a characterization
of $\dot{W}^{1,X}(\Omega)$ in spirit of \eqref{1628}.

\subsection{Characterization of $\dot{W}^{1,X}(\mathbb{R}^n)$}
\label{sub3.1}

In this subsection, we consider the case $\Omega=\mathbb{R}^n$
and establish
the following characterization
of $\dot{W}^{1,X}(\mathbb{R}^n)$.

\begin{theorem}\label{1931}
Let $X(\mathbb{R}^n)$ be a ball Banach function space,
$p\in[1,\infty)$,
and $\gamma\in\mathbb{R}\setminus\{0\}$.
Assume that
\begin{enumerate}
\item[\textup{(i)}]
$X^\frac{1}{p}(\mathbb{R}^n)$ is a ball Banach function space;
\item[\textup{(ii)}]
both $X(\mathbb{R}^n)$ and $X'(\mathbb{R}^n)$
have absolutely continuous norms;
\item[\textup{(iii)}]
the Hardy--Littlewood maximal operator $\mathcal{M}$
is bounded on $[X^\frac{1}{p}(\mathbb{R}^n)]'$;
\item[\textup{(iv)}]
if $p=1$, assume further that $\mathcal{M}$
is bounded on $X(\mathbb{R}^n)$.
\end{enumerate}
Then
$f\in\dot{W}^{1,X}(\mathbb{R}^n)$
if and only if $f\in L^1_{{\mathrm{loc}}}(\mathbb{R}^n)$ and
\begin{align}\label{2127}
\sup_{\lambda\in(0,\infty)}\lambda
\left\|\left[\int_{\mathbb{R}^n}
\mathbf{1}_{E_{\lambda,\frac{\gamma}{p}}[f]}(\cdot,y)
\left|\cdot-y\right|^{\gamma-n}\,dy
\right]^\frac{1}{p}\right\|_{X(\mathbb{R}^n)}<\infty,
\end{align}
where $E_{\lambda,\frac{\gamma}{p}}[f]$
for any $\lambda\in(0,\infty)$
is the same as in \eqref{Elambda};
moreover, if this holds true for a function $f$,
then
\begin{align*}
\sup_{\lambda\in(0,\infty)}\lambda
\left\|\left[\int_{\mathbb{R}^n}
\mathbf{1}_{E_{\lambda,\frac{\gamma}{p}}[f]}(\cdot,y)
\left|\cdot-y\right|^{\gamma-n}\,dy\right]^\frac{1}{p}
\right\|_{X(\mathbb{R}^n)}
\sim\left\|\,|\nabla f|\,\right\|_{X(\mathbb{R}^n)},
\end{align*}
where the positive equivalence constants are independent of $f$.
\end{theorem}

To show Theorem~\ref{1931},
we need some technical lemmas.
The following is a part of \cite[Theorem~3.29]{zyy2023}.

\begin{lemma}\label{1016}
Let $X(\mathbb{R}^n)$ be a ball Banach function space,
$p\in[1,\infty)$,
and $\gamma\in\mathbb{R}\setminus\{0\}$.
Assume that
\begin{enumerate}
\item[\textup{(i)}]
$X^\frac{1}{p}(\mathbb{R}^n)$ is a ball Banach function space;
\item[\textup{(ii)}]
$X(\mathbb{R}^n)$ has an absolutely continuous norm;
\item[\textup{(iii)}]
the Hardy--Littlewood maximal operator $\mathcal{M}$ is bounded on
both $X(\mathbb{R}^n)$ and $[X^\frac{1}{p}(\mathbb{R}^n)]'$.
\end{enumerate}
Then there exists a positive constant
$C$ such that,
for any $f\in\dot{W}^{1,X}(\mathbb{R}^n)$,
\begin{align*}
\sup_{\lambda\in(0,\infty)}\lambda
\left\|\left[\int_{\mathbb{R}^n}
\mathbf{1}_{E_{\lambda,\frac{\gamma}{p}}[f]}(\cdot,y)
\left|\cdot-y\right|^{\gamma-n}\,dy
\right]^\frac{1}{p}\right\|_{X(\mathbb{R}^n)}
\leq C\left\|\,\left|\nabla f\right|\,\right\|_{X(\mathbb{R}^n)},
\end{align*}
where $E_{\lambda,\frac{\gamma}{p}}[f]$
for any $\lambda\in(0,\infty)$
is the same as in \eqref{Elambda}.
\end{lemma}

In what follows,
for any given $f\in L^1_{{\mathrm{loc}}}(\mathbb{R}^n)$,
the \emph{truncations}
$\{f_{m}\}_{m\in\mathbb{N}}$ of $f$
are defined by setting, for any $m\in\mathbb{N}$
and $x\in\mathbb{R}^n$,
\begin{align}\label{fnR}
f_{m}(x):=
\begin{cases}
f(x)&\text{ if}\ |f(x)|\leq m,\\
\displaystyle
m\frac{f(x)}{|f(x)|}&\text{ if}\ |f(x)|>m.
\end{cases}
\end{align}

\begin{lemma}\label{2031}
Let $X(\mathbb{R}^n)$ be a ball Banach function space,
$p\in[1,\infty)$,
and $\gamma\in\mathbb{R}\setminus\{0\}$.
Assume that
\begin{enumerate}
\item[\textup{(i)}]
$X^\frac{1}{p}(\mathbb{R}^n)$ is a ball Banach function space;
\item[\textup{(ii)}]
both $X(\mathbb{R}^n)$ and $X'(\mathbb{R}^n)$
have absolutely continuous norms;
\item[\textup{(iii)}]
the Hardy--littlewood maximal operator $\mathcal{M}$
is bounded on $[X^\frac{1}{p}(\mathbb{R}^n)]'$.
\end{enumerate}
Then there exists a positive constant $C$ such that, for any
$f\in L^1_{{\mathrm{loc}}}(\mathbb{R}^n)$ satisfying \eqref{2127}
and $m\in\mathbb{N}$,
\begin{align*}
\left\|\,\left|\nabla f_{m}\right|\,\right\|_{X(\mathbb{R}^n)}
\leq C\sup_{\lambda\in(0,\infty)}\lambda
\left\|\left[\int_{\mathbb{R}^n}
\mathbf{1}_{E_{\lambda,\frac{\gamma}{p}}[f]}(\cdot,y)
\left|\cdot-y\right|^{\gamma-n}\,dy
\right]^\frac{1}{p}\right\|_{X(\mathbb{R}^n)},
\end{align*}
where $f_{m}$ and $E_{\lambda,\frac{\gamma}{p}}[f]$
for any $\lambda\in(0,\infty)$
are the same as, respectively, in \eqref{fnR} and \eqref{Elambda}.
\end{lemma}

To prove Lemma~\ref{2031}, we need
the following two technical lemmas.
The following extrapolation lemma is just \cite[Remark~3.9]{zyy2023bbm}.

\begin{lemma}\label{4.6}
Let $X(\mathbb{R}^n)$ be a ball Banach
function space and $p\in[1,\infty)$.
Assume that $X^\frac{1}{p}(\mathbb{R}^n)$ is a
ball Banach function space and that
the Hardy--Littlewood maximal operator
$\mathcal{M}$ is bounded on $[X^\frac{1}{p}(\mathbb{R}^n)]'$
with its operator norm denoted by
$\|\mathcal{M}\|_{[X^\frac{1}{p}(\mathbb{R}^n)]'
\to[X^\frac{1}{p}(\mathbb{R}^n)]'}$.
Assume that $\Omega\subset\mathbb{R}^n$
is an open set.
Then, for any $f\in X(\Omega)$,
\begin{align*}
\|f\|_{X(\Omega)}\leq\sup_{\|g\|_{[X^\frac{1}{p}(\mathbb{R}^n)]'}=1}
\left[\int_{\Omega}
\left|f(x)\right|^pR_{[X^\frac{1}{p}(\mathbb{R}^n)]'}
g(x)\,dx\right]^\frac{1}{p}
\leq2^\frac{1}{p}\|f\|_{X(\Omega)},
\end{align*}
where, for any $g\in[X^\frac{1}{p}(\mathbb{R}^n)]'$,
\begin{align*}
R_{[X^\frac{1}{p}(\mathbb{R}^n)]'}g:=\sum_{k=0}^\infty
\frac{\mathcal{M}^kg}{2^k\|\mathcal{M}\|^k_{
[X^\frac{1}{p}(\mathbb{R}^n)]'\to
[X^\frac{1}{p}(\mathbb{R}^n)]'}}\in A_1(\mathbb{R}^n)
\end{align*}
and
$$
\left[R_{[X^\frac{1}{p}(\mathbb{R}^n)]'}g\right]_{A_1(\mathbb{R}^n)}\leq2
\|\mathcal{M}\|_{[X^\frac{1}{p}(\mathbb{R}^n)]'\to
[X^\frac{1}{p}(\mathbb{R}^n)]'}
$$
and where, for any $k\in\mathbb{N}$,
$\mathcal{M}^k$ is the $k$ iterations of $\mathcal{M}$
and $\mathcal{M}^0g:=|g|$.
\end{lemma}

For any $R\in(0,\infty)$,
let $B_R:=B(\mathbf{0},R)$.
The following lemma can be deduced from
\cite[Corollary~5.9]{zyy2023bbm}
with $\Omega:=B_R$.

\begin{lemma}\label{2006}
Let $p\in[1,\infty)$ and $R\in(0,\infty)$.
Assume that $X(\mathbb{R}^n)$
is a ball Banach function space
satisfying the same assumptions as in Lemma~\ref{2031}.
Then $f\in W^{1,X}(B_R)$
if and only if $f\in X(B_R)$
and
\begin{align*}
\liminf_{s\to1^-}(1-s)^\frac{1}{p}\left\|\left[\int_{B_R}
\frac{|f(\cdot)-f(y)|^p}{|\cdot-y|^{n+sp}}
\,dy\right]^{\frac{1}{p}}\right\|_{X(B_R)}<\infty;
\end{align*}
moreover, for such $f$,
\begin{align*}
&\lim_{s\to1^-}(1-s)^\frac{1}{p}\left\|\left[\int_{B_R}
\frac{|f(\cdot)-f(y)|^p}{|\cdot-y|^{n+sp}}
\,dy\right]^{\frac{1}{p}}\right\|_{X(B_R)}\\
&\quad=\left[\frac{2\pi^\frac{n-1}{2}\Gamma(\frac{p+1}{2})}{p
\Gamma(\frac{p+n}{2})}\right]^\frac{1}{p}
\left\|\,\left|\nabla f\right|\,\right\|_{X(B_R)}.
\end{align*}
\end{lemma}

For any $F\in\mathscr{M}(\mathbb{R}^n\times\mathbb{R}^n)$
and $\lambda\in(0,\infty)$, let
$$
d_F(\lambda):=\left\{(x,y)\in\mathbb{R}^n\times\mathbb{R}^n:\
\left|F(x,y)\right|>\lambda\right\}.
$$
Let $\gamma\in\mathbb{R}$ and $\omega\in A_1(\mathbb{R}^n)$.
By \cite[Theorem~1.4.16(v)]{g2014},
we find that,
for any $F,G\in\mathscr{M}(\mathbb{R}^n\times\mathbb{R}^n)$
and $p\in(1,\infty)$,
\begin{align}\label{2145}
&\iint_{\mathbb{R}^n\times\mathbb{R}^n}
\left|F(x,y)G(x,y)\right||x-y|^{\gamma-n}\omega(x)\,dx\,dy\\
&\quad\leq p'\sup_{\lambda\in(0,\infty)}\lambda
\left[\mu\left(d_F(\lambda)\right)\right]^{\frac{1}{p}}
\int_0^\infty\left[\mu\left(d_G(\xi)
\right)\right]^{\frac{1}{p'}}\,d\xi\nonumber,
\end{align}
where, for any measurable set $E\subset\mathbb{R}^n\times\mathbb{R}^n$,
$$
\mu(E):=\int_{\{(x,y)\in E:\ x\neq y\}}
|x-y|^{\gamma-n}\omega(x)\,dx\,dy.
$$

Now, we are ready to show Lemma~\ref{2031}.

\begin{proof}[Proof of Lemma~\ref{2031}]
Let $f\in L^1_{{\mathrm{loc}}}(\mathbb{R}^n)$ satisfy \eqref{2127}.
Let $R\in(0,\infty)$ and let $f_{m,R}$ be
a function on $B_R$ defined by setting, for any $x\in B_R$,
$f_{m,R}(x):=f_m(x)$.
By this and both (ii) and (iv) of
\cite[Proposition~2.8]{zyy2023bbm}, we find that,
for any $m\in\mathbb{N}$ and $R\in(0,\infty)$,
$$
\left\|f_{m,R}\right\|_{X(B_R)}
\leq\|m\|_{X(B_R)}<\infty
$$
and hence $f_{m,R}\in X(B_R)$.

Next, we claim that, for any $m\in\mathbb{N}$
and $R\in(0,\infty)$,
\begin{align}\label{2011}
&\liminf_{s\to1^-}(1-s)^\frac{1}{p}\left\|\left[\int_{B_R}
\frac{|f_{m,R}(\cdot)-f_{m,R}(y)|^p}{|\cdot-y|^{n+sp}}
\,dy\right]^{\frac{1}{p}}\right\|_{X(B_R)}\\
&\quad\lesssim\sup_{\lambda\in(0,\infty)}\lambda
\left\|\left[\int_{\mathbb{R}^n}
\mathbf{1}_{E_{\lambda,\frac{\gamma}{p}}[f]}(\cdot,y)
\left|\cdot-y\right|^{\gamma-n}
\,dy\right]^\frac{1}{p}\right\|_{X(\mathbb{R}^n)}.\nonumber
\end{align}
To prove this claim,
we consider the following two cases on $\gamma$.

\emph{Case 1)} $\gamma\in(0,\infty)$. In this case,
let $s\in(0,1)$ and
$\theta:=\frac{p(1-s)}{p+\gamma}\in(0,1)$.
Then, from \eqref{2145} with $p:=\frac{1}{1-\theta}$,
it follows that,
for any $\omega\in A_1(\mathbb{R}^n)$,
\begin{align}\label{2207}
&\iint_{B_R\times B_R}
\frac{|f_{m,R}(x)-f_{m,R}(y)|^p}{|x-y|^{n+sp}}
\omega(x)\,dy\,dx\\
&\quad=\iint_{B_R\times B_R}
\left[\frac{|f_{m,R}(x)-f_{m,R}(y)|}{
|x-y|^{1+\frac{\gamma}{p}}}\right]^{p(1-\theta)}\nonumber\\
&\qquad\times\left|f_{m,R}(x)-f_{m,R}(y)\right|^{p\theta}
|x-y|^{\gamma-n}\omega(x)\,dy\,dx\nonumber\\
&\quad\leq\frac{1}{\theta}
\sup_{\lambda\in(0,\infty)}\lambda
\left[\iint_{B_R\times B_R}
\mathbf{1}_{E_{\lambda^\frac{1}{p(1-\theta)},
\frac{\gamma}{p}}[f_{m,R}]}(x,y)|x-y|^{\gamma-n}
\omega(x)\,dy\,dx\right]^{1-\theta}\nonumber\\
&\qquad\times\int_0^\infty
\Bigg[\iint_{B_R\times B_R}
\mathbf{1}_{\{(x,y):\
|f_{m,R}(x)-f_{m,R}(y)|^{p\theta}>\xi\}}(x,y)\nonumber\\
&\qquad\times|x-y|^{\gamma-n}
\omega(x)\,dy\,dx\Bigg]^{\theta}\,d\xi.\nonumber
\end{align}
Notice that, for any $x\in B_R$,
$|f_{m,R}(x)|\leq m$,
which further implies that, if
$$
\left\{(x,y)\in B_R\times B_R:\
\left|f_{m,R}(x)-f_{m,R}(y)\right|^{p\theta}
>\xi\right\}\neq\emptyset,
$$
then $\xi\in(0,(2m)^{p\theta})$.
On the other hand,
by the polar coordinate, we find that,
for any $\xi\in(0,\infty)$,
\begin{align*}
&\iint_{B_R\times B_R}
\mathbf{1}_{\{(x,y):\
|f_{m,R}(x)-f_{m,R}(y)|^{p\theta}>\xi\}}(x,y)|x-y|^{\gamma-n}
\omega(x)\,dy\,dx\\
&\quad\leq\int_{B_R}\left(\int_{B_R}
|x-y|^{\gamma-n}\,dy\right)\omega(x)\,dx\\
&\quad\leq\int_{B_R}\left(\sigma_{n-1}
\int_0^{2R}r^{\gamma-1}\,dr\right)\omega(x)\,dx\\
&\quad=\frac{\sigma_{n-1}}{\gamma}(2R)^{\gamma}
\int_{B_R}\omega(x)\,dx.
\end{align*}
From this with $\omega:=R_{[X^\frac{1}{p}(\mathbb{R}^n)]'}g$,
Lemma~\ref{4.6}, \eqref{2207},
Remark~\ref{norm},
and Definition~\ref{1659}(ii),
we infer that
\begin{align*}
&\left\|\left[\int_{B_R}
\frac{|f_{m,R}(\cdot)-f_{m,R}(y)|^p}{|\cdot-y|^{n+sp}}\,dy
\right]^{\frac{1}{p}}\right\|_{X(B_R)}^p\nonumber\\
&\quad\leq\sup_{\|g\|_{[X^\frac{1}{p}(\mathbb{R}^n)]'}=1}
\iint_{B_R\times B_R}
\frac{|f_{m,R}(x)-f_{m,R}(y)|^p}{|x-y|^{n+sp}}
R_{[X^\frac{1}{p}(\mathbb{R}^n)]'}g(x)\,dy\,dx\nonumber\\
&\quad\leq\frac{1}{\theta}\sup_{\|g\|_{[X^\frac{1}{p}(\mathbb{R}^n)]'}=1}
\sup_{\lambda\in(0,\infty)}
\lambda^{p(1-\theta)}\Bigg[\iint_{B_R\times B_R}
\mathbf{1}_{E_{\lambda,\frac{\gamma}{p}}
[f_{m,R}]}(x,y)\nonumber\\
&\qquad\times|x-y|^{\gamma-n}
R_{[X^\frac{1}{p}(\mathbb{R}^n)]'}g(x)\,dy\,dx\Bigg]^{1-\theta}\nonumber\\
&\qquad\times\int_0^{(2m)^{p\theta}}
\left[\frac{\sigma_{n-1}(2R)^{\gamma}}{\gamma}
\sup_{\|g\|_{[X^\frac{1}{p}(\mathbb{R}^n)]'}=1}
\int_{B_R}R_{[X^\frac{1}{p}(\mathbb{R}^n)]'}g(x)
\,dx\right]^{\theta}\,d\xi\nonumber\\
&\quad\leq\frac{2^{1-\theta}}{\theta}\left\{
\sup_{\lambda\in(0,\infty)}\lambda
\left\|\left[\int_{B_R}
\mathbf{1}_{E_{\lambda,\frac{\gamma}{p}}[f_{m,R}]}(\cdot,y)
\left|\cdot-y\right|^{\gamma-n}\,dy\right]^\frac{1}{p}\right\|_{X(B_R)}
\right\}^{p(1-\theta)}\nonumber\\
&\qquad\times(2m)^{p\theta}
\left[\frac{2\sigma_{n-1}(2R)^{\gamma}}{\gamma}\right]^{\theta}
\left\|1\right\|_{X(B_R)}^{p\theta}\nonumber\\
&\quad\leq\frac{2}{\theta}\left\{
\sup_{\lambda\in(0,\infty)}\lambda
\left\|\left[\int_{\mathbb{R}^n}
\mathbf{1}_{E_{\lambda,\frac{\gamma}{p}}[f]}(\cdot,y)
\left|\cdot-y\right|^{\gamma-n}\,dy
\right]^\frac{1}{p}\right\|_{X(\mathbb{R}^n)}
\right\}^{p(1-\theta)}\nonumber\\
&\qquad\times(2m)^{p\theta}
\left[\frac{\sigma_{n-1}(2R)^{\gamma}}{\gamma}\right]^{\theta}
\left\|\mathbf{1}_{B_R}\right\|_{X(\mathbb{R}^n)}^{p\theta},
\end{align*}
which, combined with both Definition~\ref{1659}(iv) and
the facts that $\theta\in(0,1)$,
$\frac{1-s}{\theta}=\frac{p+\gamma}{p}$, and
$\theta\to0$ as $s\to1^-$,
further implies that
\eqref{2011} in the
case $\gamma\in(0,\infty)$ holds true.

\emph{Case 2)} $\gamma\in(-\infty,0)$. In this case,
let $s\in(0,1)$ and
$\theta:=\frac{1-s}{2}\in(0,1)$.
Then, by \eqref{2145} with
$p:=\frac{1}{1-\theta}$, we conclude that,
for any $\omega\in A_1(\mathbb{R}^n)$,
\begin{align}\label{2013}
&\iint_{B_R\times B_R}
\frac{|f_{m,R}(x)-f_{m,R}(y)|^p}{|x-y|^{n+sp}}
\omega(x)\,dy\,dx\\
&\quad=\iint_{B_R\times B_R}
\left[\frac{|f_{m,R}(x)-f_{m,R}(y)|}{
|x-y|^{1+\frac{\gamma}{p}}}\right]^{p(1-\theta)}\nonumber\\
&\qquad\times\left[\frac{|f_{m,R}(x)-
f_{m,R}(y)|}{|x-y|^{-1+\frac{\gamma}{p}}}\right]^{p\theta}
|x-y|^{\gamma-n}\omega(x)\,dy\,dx\nonumber\\
&\quad\leq\frac{1}{\theta}
\sup_{\lambda\in(0,\infty)}\lambda
\left\{\iint_{B_R\times B_R}
\mathbf{1}_{E_{\lambda^{\frac{1}{
p(1-\theta)}},\frac{\gamma}{p}}[f_{m,R}]}(x,y)|x-y|^{\gamma-n}
\omega(x)\,dy\,dx\right\}^{1-\theta}\nonumber\\
&\qquad\times\int_0^\infty\Bigg[\iint_{B_R\times B_R}
\mathbf{1}_{\{(x,y):\ |f_{m,R}(x)-f_{m,R}(y)|^{p\theta}
|x-y|^{\theta(p-\gamma)}>\xi\}}(x,y)\nonumber\\
&\qquad\times|x-y|^{\gamma-n}
\omega(x)\,dy\,dx\Bigg]^{\theta}\,d\xi.\nonumber
\end{align}
Notice that $|f_{m,R}(x)|\leq m$ for any $x\in B_R$. Therefore, if
$$
\left\{(x,y)\in B_R\times B_R:\
\left|f_{m,R}(x)-f_{m,R}(y)\right|^{p\theta}
|x-y|^{\theta(p-\gamma)}>\xi\right\}\neq\emptyset,
$$
we then have $\xi\in(0,(2m)^{p\theta}(2R)^{\theta(p-\gamma)})$.
On the other hand,
from the polar coordinate, we deduce that,
for any $\xi\in(0,\infty)$,
\begin{align*}
&\iint_{B_R\times B_R}
\mathbf{1}_{\{(x,y):\ |f_{m,R}(x)-f_{m,R}(y)|^{p\theta}
|x-y|^{\theta(p-\gamma)}>\xi\}}(x,y)|x-y|^{\gamma-n}
\omega(x)\,dy\,dx\nonumber\\
&\quad\leq\iint_{B_R\times B_R}
\mathbf{1}_{\{(x,y):\ |x-y|^{\theta(p-\gamma)}
>\frac{\xi}{(2m)^{p\theta}}\}}(x,y)|x-y|^{\gamma-n}
\omega(x)\,dy\,dx\nonumber\\
&\quad\leq\int_{B_R}\left\{\sigma_{n-1}
\int_{[\frac{\xi}{(2m)^{p\theta}}]^{
\frac{1}{\theta(p-\gamma)}}}^\infty
r^{\gamma-1}\,dr\right\}
\omega(x)\,dx\nonumber\\
&\quad=\frac{\sigma_{n-1}}{|\gamma|}
\left[\frac{\xi}{(2m)^{p\theta}}
\right]^{\frac{\gamma}{\theta(p-\gamma)}}
\int_{B_R}\omega(x)\,dx.
\end{align*}
By this with $\omega:=R_{[X^\frac{1}{p}(\mathbb{R}^n)]'}g$,
Lemma~\ref{4.6}, \eqref{2013},
Remark~\ref{norm},
and Definition~\ref{1659}(ii),
we conclude that
\begin{align*}
&\left\|\left[\int_{B_R}
\frac{|f_{m,R}(\cdot)-f_{m,R}(y)|^p}{|\cdot-y|^{n+sp}}\,dy
\right]^{\frac{1}{p}}\right\|_{X(B_R)}^p\nonumber\\
&\quad\leq\sup_{\|g\|_{[X^\frac{1}{p}(\mathbb{R}^n)]'}=1}
\iint_{B_R\times B_R}
\frac{|f_{m,R}(x)-f_{m,R}(y)|^p}{|x-y|^{n+sp}}
R_{[X^\frac{1}{p}(\mathbb{R}^n)]'}g(x)\,dy\,dx\nonumber\\
&\quad\leq\frac{2^{1-\theta}}{\theta}
\left\{\sup_{\lambda\in(0,\infty)}\lambda
\left\|\left[\int_{B_R}
\mathbf{1}_{E_{\lambda,\frac{\gamma}{p}}[f_{m,R}]}(\cdot,y)
\left|\cdot-y\right|^{\gamma-n}\,dy\right]^\frac{1}{p}
\right\|_{X(B_R)}
\right\}^{p(1-\theta)}\nonumber\\
&\qquad\times\sup_{\|g\|_{[X^\frac{1}{p}(\mathbb{R}^n)]'}=1}
\left\{\int_{B_R}R_{[X^\frac{1}{p}(\mathbb{R}^n)]'}g(x)
\,dx\right\}^\theta\nonumber\\
&\qquad\times\left(\frac{\sigma_{n-1}}{|\gamma|}\right)^\theta
\left[\frac{1}{(2m)^{p\theta}}
\right]^{\frac{\gamma}{p-\gamma}}
\int_0^{(2m)^{p\theta}
(2R)^{\theta(p-\gamma)}}
\xi^{\frac{\gamma}{p-\gamma}}\,d\xi
\nonumber\\
&\quad\leq\frac{2}{\theta}
\left\{\sup_{\lambda\in(0,\infty)}\lambda
\left\|\left[\int_{\mathbb{R}^n}
\mathbf{1}_{E_{\lambda,\frac{\gamma}{p}}[f]}(\cdot,y)
\left|\cdot-y\right|^{\gamma-n}\,dy
\right]^\frac{1}{p}\right\|_{X(\mathbb{R}^n)}
\right\}^{p(1-\theta)}\nonumber\\
&\qquad\times\frac{p-\gamma}{p}\left(4mR\right)^{p\theta}
\left(\frac{\sigma_{n-1}}{|\gamma|}\right)^{\theta}
\|\mathbf{1}_{B_R}\|_{X(\mathbb{R}^n)}^{p\theta},
\end{align*}
which, together with Definition~\ref{1659}(iv) and
the facts that $\theta\in(0,1)$,
$\frac{1-s}{\theta}=2$, and
$\theta\to0$ as $s\to1^-$,
further implies that \eqref{2011}
in the case $\gamma\in(-\infty,0)$ holds true
and hence completes the proof of the above claim.

From the above claim, the above proven
conclusion that $f_{m,R}\in X(B_R)$,
and Lemma~\ref{2006},
we infer that, for any $m\in\mathbb{N}$ and $R\in(0,\infty)$,
$|\nabla f_{m,R}|\in X(B_R)$ and
\begin{align}\label{1032}
\left\|\,\left|\nabla f_{m,R}\right|\,\right\|_{X(B_R)}
&\sim\lim_{s\to1^-}(1-s)^\frac{1}{p}\left\|\left[\int_{B_R}
\frac{|f_{m}(\cdot)-f_{m}(y)|^p}{|\cdot-y|^{n+sp}}
\,dy\right]^{\frac{1}{p}}\right\|_{X(B_R)}\\
&\lesssim\sup_{\lambda\in(0,\infty)}\lambda
\left\|\left[\int_{\mathbb{R}^n}
\mathbf{1}_{E_{\lambda,\frac{\gamma}{p}}[f]}(\cdot,y)
\left|\cdot-y\right|^{\gamma-n}\,dy
\right]^\frac{1}{p}\right\|_{X(\mathbb{R}^n)},\nonumber
\end{align}
where the implicit positive constants are
independent of both $m$ and $R$.
On the other hand, by the definitions of
both $f_{m,R}$ and the weak derivatives,
we are easy to show that, for any given $0<R_1<R_2<\infty$
and for almost every $x\in B_{R_1}$,
\begin{align}\label{2242}
\nabla f_{m,R_1}(x)
=\nabla f_{m,R_2}(x).
\end{align}
Thus, we find that, for any $j\in\{1,\ldots,n\}$
and for almost every $x\in\mathbb{R}^n$,
\begin{align*}
f^{(j)}_m(x):=\lim_{R\in\mathbb{N},\,
R\to\infty}\partial_jf_{m,R}(x)
\end{align*}
is well defined.
Moreover, from \eqref{2242},
it is easy to prove that, for almost every $x\in B_R$
with $R\in\mathbb{N}$,
\begin{align}\label{1036}
f^{(j)}_m(x)=\partial_jf_{m,R}(x).
\end{align}

Now, we show that,
for any $j\in\{1,\ldots,n\}$,
$\partial_jf_m$ exists and is equal to $f^{(j)}_m$
almost everywhere in $\mathbb{R}^n$.
Indeed, by the definitions of
both $f_{m,R}$ and the weak derivatives
and \eqref{1036}, we conclude that,
for any $j\in\{1,\ldots,n\}$ and
$\phi\in C_{\mathrm{c}}^\infty(\mathbb{R}^n)$
with $\mathrm{supp\,}(\phi)\subset B_R$ for some $R\in\mathbb{N}$,
\begin{align*}
\int_{\mathbb{R}^n}f_m(x)\partial_j\phi(x)\,dx
&=\int_{B_R}f_{m,R}(x)\partial_j\phi(x)\,dx
=-\int_{B_R}\partial_jf_{m,R}(x)\phi(x)\,dx\\
&=-\int_{B_R}f^{(j)}_m(x)\phi(x)\,dx
=-\int_{\mathbb{R}^n}f^{(j)}_m(x)\phi(x)\,dx.
\end{align*}
From this, Definition~\ref{1659}(iii), and \eqref{1032},
we deduce that
\begin{align*}
\left\|\,\left|\nabla f_m\right|\,\right\|_{X(\mathbb{R}^n)}
&\sim\sum_{j=1}^n\left\|f_m^{(j)}\right\|_{X(\mathbb{R}^n)}
=\sum_{j=1}^n\lim_{R\in\mathbb{N},\,
R\to\infty}\left\|\widetilde{\partial_jf_{m,R}}
\right\|_{X(\mathbb{R}^n)}\\
&=\sum_{j=1}^n\lim_{R\in\mathbb{N},\,
R\to\infty}\left\|\partial_jf_{m,R}\right\|_{X(B_R)}\\
&\sim\lim_{R\in\mathbb{N},\,R\to\infty}
\left\|\,\left|\nabla f_{m,R}\right|\,\right\|_{X(B_R)}\\
&\lesssim\sup_{\lambda\in(0,\infty)}\lambda
\left\|\left[\int_{\mathbb{R}^n}
\mathbf{1}_{E_{\lambda,\frac{\gamma}{p}}[f]}(\cdot,y)
\left|\cdot-y\right|^{\gamma-n}\,dy
\right]^\frac{1}{p}\right\|_{X(\mathbb{R}^n)},
\end{align*}
where $\widetilde{\partial_jf_{m,R}}$
is defined the same as in \eqref{1448}
with $f$ replaced by $\partial_jf_{m,R}$,
which completes the proof of Lemma~\ref{2031}.
\end{proof}

The following lemma is a generalization
of \cite[Chapter 1, Corollary~4.4]{bs1988},
whose proof remains true for ball Banach function spaces;
we omit the details.

\begin{lemma}\label{reflexive}
Let $X(\mathbb{R}^n)$ be a ball Banach function space.
Let $\Omega\subset\mathbb{R}^n$ be an open set
and $X(\Omega)$ the restrictive space of $X(\mathbb{R}^n)$ on $\Omega$.
Then $X(\Omega)$ is reflexive if and only if
both $X(\Omega)$ and $X'(\Omega)$
have absolutely continuous norms on $\Omega$.
\end{lemma}

The following lemma is very useful, which shows that,
for any ball Banach function space $X(\mathbb{R}^n)$,
the boundedness of the Hardy--Littlewood
maximal operator on the associate space
of its convexification implies the boundedness
of $\mathcal{M}$ on $X(\mathbb{R}^n)$ itself
(or its convexification);
the converse may not be true.

\begin{lemma}\label{2005}
Let $X(\mathbb{R}^n)$ be a ball Banach function space
and $p\in[1,\infty)$.
Assume that $X^\frac{1}{p}(\mathbb{R}^n)$ is a
ball Banach function space and that
the Hardy--Littlewood maximal operator
$\mathcal{M}$ is bounded on $[X^\frac{1}{p}(\mathbb{R}^n)]'$.
\begin{enumerate}
\item[\rm(i)]
If $p\in(1,\infty)$, then
$\mathcal{M}$ is bounded on $X(\mathbb{R}^n)$.
\item[\rm(ii)]
If $p=1$, then, for any $\theta\in(1,\infty)$,
$\mathcal{M}$ is bounded on $X^\theta(\mathbb{R}^n)$.
\end{enumerate}
\end{lemma}

\begin{proof}
From Remark~\ref{920}(i),
we infer that $X(\mathbb{R}^n)$ is also
a Banach function space in Definition~\ref{016}.
By this and the equivalence of both (i) and (iii) of
\cite[Theorem~3.1]{lz2023},
we conclude that (i) of the present lemma holds true.

Next, we show (ii).
From Definition~\ref{tuhua},
Lemma~\ref{4.6} with $p=1$, and
both (iii) and (i) of Lemma~\ref{ApProperty},
we deduce that, for any given $\theta\in(1,\infty)$
and for any $f\in X^\theta(\mathbb{R}^n)$,
\begin{align*}
\left\|\mathcal{M}(f)\right\|_{X^\theta(\mathbb{R}^n)}
&=\left\|\left[\mathcal{M}(f)\right]^\theta
\right\|_{X(\mathbb{R}^n)}^\frac{1}{\theta}\\
&\leq\sup_{\|g\|_{[X(\mathbb{R}^n)]'}=1}
\left\{\int_{\mathbb{R}^n}
\left[\mathcal{M}(f)(x)\right]^\theta
R_{[X(\mathbb{R}^n)]'}
g(x)\,dx\right\}^\frac{1}{\theta}\\
&\lesssim\sup_{\|g\|_{[X(\mathbb{R}^n)]'}=1}
\left[R_{[X(\mathbb{R}^n)]'}g\right]_{A_\theta(\mathbb{R}^n)}^{\theta'-1}
\left\{\int_{\mathbb{R}^n}
\left|f(x)\right|^\theta
R_{[X(\mathbb{R}^n)]'}
g(x)\,dx\right\}^\frac{1}{\theta}\\
&\leq\sup_{\|g\|_{[X(\mathbb{R}^n)]'}=1}
\left[R_{[X(\mathbb{R}^n)]'}g\right]_{A_1(\mathbb{R}^n)}^{\theta'-1}
\left\{\int_{\mathbb{R}^n}
\left|f(x)\right|^\theta
R_{[X(\mathbb{R}^n)]'}
g(x)\,dx\right\}^\frac{1}{\theta}\\
&\lesssim\|\mathcal{M}\|_{[X(\mathbb{R}^n)]'
\to[X(\mathbb{R}^n)]'}^{\theta'-1}
\sup_{\|g\|_{[X(\mathbb{R}^n)]'}=1}
\left\{\int_{\mathbb{R}^n}
\left|f(x)\right|^\theta
R_{[X(\mathbb{R}^n)]'}
g(x)\,dx\right\}^\frac{1}{\theta}\\
&\lesssim\|\mathcal{M}\|_{[X(\mathbb{R}^n)]'
\to[X(\mathbb{R}^n)]'}^{\theta'-1}
\left\|f\right\|_{X^\theta(\mathbb{R}^n)}.
\end{align*}
This finishes the proof of (ii) and hence Lemma~\ref{2005}.
\end{proof}

Now, we are ready to prove Theorem~\ref{1931}.

\begin{proof}[Proof of Theorem~\ref{1931}]
We first show the necessity.
Let $f\in\dot{W}^{1,X}(\mathbb{R}^n)$.
Then $f\in L^1_{{\mathrm{loc}}}(\mathbb{R}^n)$.
Using Lemma~\ref{2005}(i),
we conclude that, when $p\in(1,\infty)$,
$\mathcal{M}$ is bounded on $X(\mathbb{R}^n)$,
which, combined with the assumption (iv) of the present
theorem, further implies that
$\mathcal{M}$ is bounded on $X(\mathbb{R}^n)$
whenever $p\in[1,\infty)$.
From this, Lemma~\ref{1016},
and Definition~\ref{2.7}(i), we infer that
\begin{align*}
\sup_{\lambda\in(0,\infty)}\lambda
\left\|\left[\int_{\mathbb{R}^n}
\mathbf{1}_{E_{\lambda,\frac{\gamma}{p}}[f]}(\cdot,y)
\left|\cdot-y\right|^{\gamma-n}\,dy
\right]^\frac{1}{p}\right\|_{X(\mathbb{R}^n)}
\lesssim\left\|\,|\nabla f|\,\right\|_{X(\mathbb{R}^n)}<\infty.
\end{align*}
This finishes the proof of the necessity.

Next, we prove the sufficiency.
To this end, let $f\in L^1_{{\mathrm{loc}}}
(\mathbb{R}^n)$ satisfy \eqref{2127}.
By Lemma~\ref{2031}, we find that, for any $j\in\{1,\ldots,n\}$
and $m\in\mathbb{N}$,
\begin{align}\label{1028}
\left\|\partial_jf_{m}\right\|_{X(\mathbb{R}^n)}
&\leq\left\|\,\left|\nabla f_{m}\right|\,\right\|_{X(\mathbb{R}^n)}\\
&\lesssim\sup_{\lambda\in(0,\infty)}\lambda
\left\|\left[\int_{\mathbb{R}^n}
\mathbf{1}_{E_{\lambda,\frac{\gamma}{p}}[f]}(\cdot,y)
\left|\cdot-y\right|^{\gamma-n}\,dy
\right]^\frac{1}{p}\right\|_{X(\mathbb{R}^n)},\nonumber
\end{align}
where the implicit positive constant is
independent of both $f$ and $m$.
Using the assumption
(ii) of the present theorem,
both (ii) and (iii) of Remark~\ref{dual},
and Lemma~\ref{reflexive},
we conclude that
\begin{align*}
X'(\mathbb{R}^n)=X^*(\mathbb{R}^n)
\quad\text{and}\quad
X(\mathbb{R}^n)=X''(\mathbb{R}^n)
=X^{**}(\mathbb{R}^n),
\end{align*}
which, combined with both \eqref{1028} and the Alaoglu theorem
(see, for instance, \cite[Theorem 3.17]{r91}),
further implies that there exist 
$g_j\in X(\mathbb{R}^n)$ and
$\{m_k\}_{k\in\mathbb{N}}\subset\mathbb{N}$
satisfying that
$m_k\to\infty$ as $k\to\infty$
such that, for any $\phi\in X'(\mathbb{R}^n)$,
\begin{align}\label{2154}
\int_{\mathbb{R}^n}\partial_jf_{m_k}(x)\phi(x)\,dx\to
\int_{\mathbb{R}^n}g_j(x)\phi(x)\,dx
\end{align}
as $k\to\infty$.
From this and \eqref{1028},
it follows that, for any $j\in\{1,\ldots,n\}$,
\begin{align}\label{2050}
\left\|g_j\right\|_{X(\mathbb{R}^n)}
&=\sup_{\|\phi\|_{X'(\mathbb{R}^n)}=1}
\left|\int_{\mathbb{R}^n}g_j(x)\phi(x)\,dx\right|\\
&=\sup_{\|\phi\|_{X'(\mathbb{R}^n)}=1}
\lim_{k\to\infty}
\left|\int_{\mathbb{R}^n}\partial_j
f_{m_k}(x)\phi(x)\,dx\right|\nonumber\\
&\leq\sup_{\|\phi\|_{X'(\mathbb{R}^n)}=1}
\sup_{m\in\mathbb{N}}
\left|\int_{\mathbb{R}^n}\partial_j
f_{m}(x)\phi(x)\,dx\right|
=\sup_{m\in\mathbb{N}}
\left\|\partial_jf_{m}\right\|_{X(\mathbb{R}^n)}\nonumber\\
&\lesssim\sup_{\lambda\in(0,\infty)}\lambda
\left\|\left[\int_{\mathbb{R}^n}
\mathbf{1}_{E_{\lambda,\frac{\gamma}{p}}[f]}(\cdot,y)
\left|\cdot-y\right|^{\gamma-n}\,dy
\right]^\frac{1}{p}\right\|_{X(\mathbb{R}^n)}.\nonumber
\end{align}
On the other hand, by Remarks~\ref{dual}(i)
and~\ref{1052}(v), we find that
$C_{\mathrm{c}}^\infty(\mathbb{R}^n)\subset X'(\mathbb{R}^n)$.
From this,
the Lebesgue dominated convergence theorem,
$f\in L^1_{{\mathrm{loc}}}(\mathbb{R}^n)$, and \eqref{2154},
we deduce that, for any $j\in\{1,\ldots,n\}$ and
$\phi\in C_{\mathrm{c}}^\infty(\mathbb{R}^n)$,
\begin{align*}
\int_{\mathbb{R}^n}f(x)\partial_j\phi(x)\,dx
&=\lim_{k\to\infty}
\int_{\mathbb{R}^n}f_{m_k}(x)\partial_j\phi(x)\,dx\\
&=-\lim_{k\to\infty}
\int_{\mathbb{R}^n}\partial_jf_{m_k}(x)\phi(x)\,dx\\
&=-\int_{\mathbb{R}^n}g_j(x)\phi(x)\,dx,
\end{align*}
which, combined with both the definition
and the uniqueness of weak derivatives
(see, for instance, \cite[pp.\,143--144]{eg2015}),
further implies that $\partial_jf$ exists and
$\partial_jf=g_j$ almost everywhere in $\mathbb{R}^n$.
Combining this and \eqref{2050}, we conclude that
\begin{align*}
\left\|\,|\nabla f|\,\right\|_{X(\mathbb{R}^n)}
&\sim\sum_{j=1}^n\left\|g_j\right\|_{X(\mathbb{R}^n)}\\
&\lesssim\sup_{\lambda\in(0,\infty)}\lambda
\left\|\left[\int_{\mathbb{R}^n}
\mathbf{1}_{E_{\lambda,\frac{\gamma}{p}}[f]}(\cdot,y)
\left|\cdot-y\right|^{\gamma-n}\,dy
\right]^\frac{1}{p}\right\|_{X(\mathbb{R}^n)}<\infty
\end{align*}
and hence $f\in\dot{W}^{1,X}(\mathbb{R}^n)$.
This finishes the proof of the sufficiency
and hence Theorem~\ref{1931}.
\end{proof}

\begin{remark}
On the assumption (iii) of Theorem~\ref{1931},
if $p\in(1,\infty)$, then Lorist and Nieraeth \cite[Theorem~3.1]{lz2023}
gave two equivalent characterizations on the boundedness of
the Hardy--Littlewood maximal operator $\mathcal M$ on
$[X^{\frac 1p}({\mathbb R}^n)]'$, which further implies
that $\mathcal M$ is bounded on both $X(\mathbb{R}^n)$ and $X'(\mathbb{R}^n)$.
\end{remark}

Following the proof of Theorem~\ref{1931} with Lemma~\ref{1016}
replaced by \cite[Theorem~3.29]{zyy2023},
we obtain the following conclusion
which does not need the boundedness
assumption of the Hardy--Littlewood
maximal operator on $X(\mathbb{R}^n)$;
we omit the details here.

\begin{proposition}\label{3.8}
Let $X(\mathbb{R}^n)$ be a ball Banach function space.
Assume that
\begin{enumerate}
\item[\textup{(i)}]
both $X(\mathbb{R}^n)$ and $X'(\mathbb{R}^n)$
have absolutely continuous norms;
\item[\textup{(ii)}]
the Hardy--Littlewood maximal operator $\mathcal{M}$
is bounded on $X'(\mathbb{R}^n)$;
\item[\textup{(iii)}]
$\gamma\in(0,\infty)$ or both $\gamma\in(-\infty,-1)$ and $n=1$.
\end{enumerate}
Then
$f\in\dot{W}^{1,X}(\mathbb{R}^n)$
if and only if $f\in L^1_{{\mathrm{loc}}}(\mathbb{R}^n)$ and
\begin{align*}
\sup_{\lambda\in(0,\infty)}\lambda
\left\|\int_{\mathbb{R}^n}
\mathbf{1}_{E_{\lambda,\gamma}[f]}(\cdot,y)
\left|\cdot-y\right|^{\gamma-n}\,dy\right\|_{X(\mathbb{R}^n)}<\infty,
\end{align*}
where $E_{\lambda,\gamma}[f]$
for any $\lambda\in(0,\infty)$
is the same as in \eqref{Elambda}
with $\gamma/p$ replaced by $\gamma$;
moreover, if this holds true for a function $f$,
then
\begin{align*}
\sup_{\lambda\in(0,\infty)}\lambda
\left\|\int_{\mathbb{R}^n}
\mathbf{1}_{E_{\lambda,\gamma}[f]}(\cdot,y)
\left|\cdot-y\right|^{\gamma-n}\,dy\right\|_{X(\mathbb{R}^n)}
\sim\left\|\,|\nabla f|\,\right\|_{X(\mathbb{R}^n)},
\end{align*}
where the positive equivalence constants are independent of $f$.
\end{proposition}

\subsection{Characterization of $\dot{W}^{1,X}(\Omega)$}
\label{sub3.2}

This subsection is devoted to establishing
a characterization
of $\dot{W}^{1,X}(\Omega)$,
where $\Omega\subset\mathbb{R}^n$
is an $(\varepsilon,\infty)$-domain for some $\varepsilon\in(0,1]$.
To this end, we begin with recalling the concept
of $(\varepsilon,\infty)$-domains
which was originally
introduced by Martio and Sarvas \cite{ms1979}
in connection with approximation and
injectivity properties of mappings.

\begin{definition}\label{2121}
Let $\varepsilon\in(0,1]$.
An open set $\Omega\subset\mathbb{R}^n$
is called an \emph{$(\varepsilon,\infty)$-domain}
if, for any $x,y\in\Omega$ with $x\neq y$,
there exists a rectifiable curve $\Gamma$ such that
\begin{enumerate}
\item[\textup{(i)}]
$\Gamma$ lies in $\Omega$;
\item[\textup{(ii)}]
$\Gamma$ connects $x$ and $y$;
\item[\textup{(iii)}]
$\ell(\Gamma)\leq\varepsilon^{-1}|x-y|$,
where $\ell(\Gamma)$ denotes
the Euclidean arc length of $\Gamma$;
\item[\textup{(iv)}]
for any point $z$ on $\Gamma$,
\begin{align*}
\mathrm{dist\,}(z,\partial\Omega)\ge\frac{\varepsilon|x-z||y-z|}{|x-y|}.
\end{align*}
\end{enumerate}
\end{definition}

\begin{remark}\label{1638}
\begin{enumerate}
\item[\textup{(i)}]
As was pointed out in \cite[p.\,276]{dhhr2011},
$(\varepsilon,\infty)$-domains are also known
as \emph{uniform domains} \cite{ms1979}
or \emph{Jones domains} \cite{j1981}.
We refer the reader to \cite{ge,gm1985,go1979,
hk1991,hk1992,m80,v3} for more
studies of uniform domains.
\item[\textup{(ii)}]
As was pointed out in \cite[p.\,73]{j1981},
the classical snowflake domain
(see, for instance, \cite[pp.\,104--105]{lv1973})
is an $(\varepsilon,\infty)$-domain
for some $\varepsilon\in(0,1]$.
Also, as was pointed out in \cite[p.\,276]{dhhr2011},
any bounded Lipschitz domain
and the half-space are $(\varepsilon,\infty)$-domains
for some $\varepsilon\in(0,1]$.
\end{enumerate}
\end{remark}

Now, we present the main theorem of this subsection as follows.

\begin{theorem}\label{1039}
Let $\Omega\subset\mathbb{R}^n$ be an $(\varepsilon,\infty)$-domain
with $\varepsilon\in(0,1]$,
$p\in[1,\infty)$,
and $\gamma\in\mathbb{R}\setminus\{0\}$.
Assume that
$X(\mathbb{R}^n)$ is a ball Banach function space
satisfying the same assumptions as in Theorem~\ref{1931}.
Then
$f\in\dot{W}^{1,X}(\Omega)$
if and only if $f\in L^1_{{\mathrm{loc}}}(\Omega)$ and
\begin{align}\label{1040}
\sup_{\lambda\in(0,\infty)}\lambda
\left\|\left[\int_{\Omega}
\mathbf{1}_{E_{\lambda,\frac{\gamma}{p}}[f]}(\cdot,y)
\left|\cdot-y\right|^{\gamma-n}\,dy
\right]^\frac{1}{p}\right\|_{X(\Omega)}<\infty,
\end{align}
where $E_{\lambda,\frac{\gamma}{p}}[f]$
for any $\lambda\in(0,\infty)$
is the same as in \eqref{Elambda};
moreover, if this holds true for a function $f$,
then
\begin{align*}
\sup_{\lambda\in(0,\infty)}\lambda
\left\|\left[\int_{\Omega}
\mathbf{1}_{E_{\lambda,\frac{\gamma}{p}}[f]}(\cdot,y)
\left|\cdot-y\right|^{\gamma-n}\,dy\right]^\frac{1}{p}
\right\|_{X(\Omega)}
\sim\left\|\,|\nabla f|\,\right\|_{X(\Omega)},
\end{align*}
where the positive equivalence constants are independent of $f$.
\end{theorem}

To show Theorem~\ref{1039},
we need the following extension lemma
on $\dot{W}^{1,X}(\Omega)$,
which can be found in \cite[Theorem~3.13]{zyy2023bbm}.

\begin{lemma}\label{extension2}
Let $\Omega\subset\mathbb{R}^n$ be an $(\varepsilon,\infty)$-domain
with $\varepsilon\in(0,1]$
and $X(\mathbb{R}^n)$ a ball Banach function space
having an absolutely continuous norm.
Assume that
there exists $p\in[1,\infty)$ such that
$X^\frac{1}{p}(\mathbb{R}^n)$
is a ball Banach function space
and that the Hardy--Littlewood maximal operator
is bounded on $[X^\frac{1}{p}(\mathbb{R}^n)]'$.
Then there exists a linear extension operator
$$
\Lambda:\ \dot{W}^{1,X}(\Omega)\to\dot{W}^{1,X}(\mathbb{R}^n)
$$
such that,
for any $f\in\dot{W}^{1,X}(\Omega)$,
\begin{enumerate}
\item[\textup{(i)}]
$\Lambda(f)=f$ almost everywhere in $\Omega$;
\item[\textup{(ii)}]
$\|\Lambda(f)\|_{\dot{W}^{1,X}(\mathbb{R}^n)}
\leq C\|f\|_{\dot{W}^{1,X}(\Omega)}$,
where the positive constant $C$ is independent of $f$.
\end{enumerate}
\end{lemma}

\begin{proof}[Proof of Theorem~\ref{1039}]
We first prove the necessity. Let $f\in\dot{W}^{1,X}(\Omega)$.
Then, by Lemma~\ref{extension2}, we find that
there exists $g\in\dot{W}^{1,X}(\mathbb{R}^n)$
such that $g=f$ almost everywhere in $\Omega$ and
$$
\left\|\,|\nabla g|\,\right\|_{X(\mathbb{R}^n)}
\lesssim\left\|\,|\nabla f|\,\right\|_{X(\Omega)}.
$$
From this and Theorem~\ref{1931} with $f:=g$,
we infer that
\begin{align*}
&\sup_{\lambda\in(0,\infty)}\lambda
\left\|\left[\int_{\Omega}
\mathbf{1}_{E_{\lambda,\frac{\gamma}{p}}[f]}(\cdot,y)
\left|\cdot-y\right|^{\gamma-n}\,dy\right]^\frac{1}{p}
\right\|_{X(\Omega)}\\
&\quad=\sup_{\lambda\in(0,\infty)}\lambda
\left\|\left[\int_{\Omega}
\mathbf{1}_{E_{\lambda,\frac{\gamma}{p}}[g]}(\cdot,y)
\left|\cdot-y\right|^{\gamma-n}\,dy\right]^\frac{1}{p}
\right\|_{X(\Omega)}\\
&\quad\leq\sup_{\lambda\in(0,\infty)}\lambda
\left\|\left[\int_{\mathbb{R}^n}
\mathbf{1}_{E_{\lambda,\frac{\gamma}{p}}[g]}(\cdot,y)
\left|\cdot-y\right|^{\gamma-n}\,dy\right]^\frac{1}{p}
\right\|_{X(\mathbb{R}^n)}\\
&\quad\sim\left\|\,|\nabla g|\,\right\|_{X(\mathbb{R}^n)}
\lesssim\left\|\,|\nabla f|\,\right\|_{X(\Omega)},
\end{align*}
which completes the proof of the necessity.

Next, we show the sufficiency.
Let $f\in L_{\mathrm{loc}}^1(\Omega)$ satisfy \eqref{1040}.
For any $m\in\mathbb{N}$, let
$f_m$ be defined the same as in \eqref{fnR}.
Following the proof of Lemma~\ref{2031}
with $\mathbb{R}^n$, $B_R$, and Lemma~\ref{2006}
replaced, respectively, by $\Omega$, $B_R\cap\Omega$, and
\cite[Corollary~5.9]{zyy2023bbm} with $\Omega:=B_R\cap\Omega$,
we conclude that, for any $m\in\mathbb{N}$,
\begin{align*}
\left\|\,\left|\nabla f_{m}\right|\,\right\|_{X(\Omega)}
\lesssim\sup_{\lambda\in(0,\infty)}\lambda
\left\|\left[\int_{\Omega}
\mathbf{1}_{E_{\lambda,\frac{\gamma}{p}}[f]}(\cdot,y)
\left|\cdot-y\right|^{\gamma-n}\,dy
\right]^\frac{1}{p}\right\|_{X(\Omega)},
\end{align*}
where the implicit positive constant is independent of $m$.
From this and an argument similar to that used in the proof of
the sufficiency of Theorem~\ref{1931}
with Remarks~\ref{dual}(i)
and~\ref{1052}(v) replaced by both
(i) and (iv) of Remark~\ref{dual} and Remark~\ref{norm}(iii),
we deduce that, for any $j\in\{1,\ldots,n\}$,
there exists $g_j\in X(\Omega)$ such that, for any
$\phi\in C_{\mathrm{c}}^\infty(\Omega)$,
\begin{align*}
\int_{\Omega}f(x)\partial_j\phi(x)\,dx
=-\int_{\Omega}g_j(x)\phi(x)\,dx
\end{align*}
and
\begin{align}\label{1713}
\left\|g_j\right\|_{X(\Omega)}
\lesssim\sup_{\lambda\in(0,\infty)}\lambda
\left\|\left[\int_{\Omega}
\mathbf{1}_{E_{\lambda,\frac{\gamma}{p}}[f]}(\cdot,y)
\left|\cdot-y\right|^{\gamma-n}\,dy
\right]^\frac{1}{p}\right\|_{X(\Omega)}.
\end{align}
By this, we conclude that, for any $j\in\{1,\ldots,n\}$,
$\partial_jf$ exists and
$\partial_jf=g_j$ almost everywhere in $\Omega$,
which, combined with \eqref{1713}, further implies that
\begin{align*}
\left\|\,|\nabla f|\,\right\|_{X(\Omega)}
&\sim\sum_{j=1}^n\left\|g_j\right\|_{X(\Omega)}\\
&\lesssim\sup_{\lambda\in(0,\infty)}\lambda
\left\|\left[\int_{\Omega}
\mathbf{1}_{E_{\lambda,\frac{\gamma}{p}}[f]}(\cdot,y)
\left|\cdot-y\right|^{\gamma-n}\,dy
\right]^\frac{1}{p}\right\|_{X(\Omega)}<\infty
\end{align*}
and hence $f\in\dot{W}^{1,X}(\Omega)$.
This finishes the proof of the sufficiency
and hence Theorem~\ref{1039}.
\end{proof}

Following the proof of Theorem~\ref{1039} with Theorem~\ref{1931}
replaced by Proposition~\ref{3.8},
we obtain the following conclusion
which does not need the boundedness
assumption of the Hardy--Littlewood
maximal operator on $X(\mathbb{R}^n)$;
we omit the details here.

\begin{proposition}
Let $\Omega\subset\mathbb{R}^n$ be
an $(\varepsilon,\infty)$-domain
with $\varepsilon\in(0,1]$.
Assume that
$X(\mathbb{R}^n)$ is a ball Banach function space
satisfying the same assumptions as in
Proposition~\ref{3.8}.
Then
$f\in\dot{W}^{1,X}(\Omega)$
if and only if $f\in L^1_{{\mathrm{loc}}}(\Omega)$ and
\begin{align*}
\sup_{\lambda\in(0,\infty)}\lambda
\left\|\int_{\Omega}
\mathbf{1}_{E_{\lambda,\gamma}[f]}(\cdot,y)
\left|\cdot-y\right|^{\gamma-n}\,dy\right\|_{X(\Omega)}<\infty,
\end{align*}
where $E_{\lambda,\gamma}[f]$
for any $\lambda\in(0,\infty)$
is the same as in \eqref{Elambda}
with $\gamma/p$ replaced by $\gamma$;
moreover, if this holds true for a function $f$,
then
\begin{align*}
\sup_{\lambda\in(0,\infty)}\lambda
\left\|\int_{\Omega}
\mathbf{1}_{E_{\lambda,\gamma}[f]}(\cdot,y)
\left|\cdot-y\right|^{\gamma-n}\,dy\right\|_{X(\Omega)}
\sim\left\|\,|\nabla f|\,\right\|_{X(\Omega)},
\end{align*}
where the positive equivalence constants are independent of $f$.
\end{proposition}

\section{Applications to Specific Function Spaces}
\label{S5}

In this section, we apply Theorems~\ref{1931} and~\ref{1039}
to ten examples of ball
Banach function spaces, namely
weighted Lebesgue spaces and
Morrey spaces (see Subsection~\ref{5.4} below),
Besov--Bourgain--Morrey spaces
(see Subsection~\ref{BBMspace} below),
local generalized Herz spaces
and global generalized Herz spaces (see Subsection~\ref{Herz} below),
mixed-norm Lebesgue spaces (see Subsection~\ref{5.2} below),
variable Lebesgue spaces (see Subsection~\ref{5.3} below),
Lorentz spaces (see Subsection~\ref{5.7} below),
Orlicz spaces (see Subsection~\ref{5.5} below),
and Orlicz-slice spaces (see Subsection~\ref{5.6} below).
Throughout this section,
we always assume that $\Omega=\mathbb{R}^n$
or $\Omega\subset\mathbb{R}^n$ is an $(\varepsilon,\infty)$-domain
for some $\varepsilon\in(0,1]$.

\subsection{Weighted Lebesgue Spaces
and Morrey Spaces}\label{5.4}

We first consider the case of weighted Lebesgue spaces;
see \eqref{2040} for their definitions.
As was pointed out in \cite[p.\,86]{shyy2017},
the weighted Lebesgue space $L^r_\omega(\mathbb{R}^n)$
is a ball quasi-Banach function space,
but it may not be a Banach function space
in the sense of Bennett and Sharpley \cite{bs1988}.
When $X:=L^r_{\omega}$,
we simply write $\dot{W}^{1,r}_{\omega}
(\Omega):=\dot{W}^{1,X}(\Omega)$.

Using Theorems~\ref{1931} and~\ref{1039},
we obtain the following conclusion.

\begin{theorem}\label{2044}
Let $\gamma\in\mathbb{R}\setminus\{0\}$,
$r\in(1,\infty)$, $p\in[1,r]$, and
$\omega\in A_{r/p}(\mathbb{R}^n)$.
Then $f\in\dot{W}^{1,r}_\omega(\Omega)$ if and only if
$f\in L_{{\mathrm{loc}}}^1(\Omega)$ and
$$
\sup_{\lambda\in(0,\infty)}\lambda
\left\|\left[\int_{\Omega}
\mathbf{1}_{E_{\lambda,\frac{\gamma}{p}}[f]}(\cdot,y)
\left|\cdot-y\right|^{\gamma-n}\,dy\right]^\frac{1}{p}
\right\|_{L^{r}_\omega(\Omega)}<\infty,
$$
where $E_{\lambda,\frac{\gamma}{p}}[f]$
for any $\lambda\in(0,\infty)$
is the same as in \eqref{Elambda};
moreover, for such $f$,
$$
\sup_{\lambda\in(0,\infty)}\lambda
\left\|\left[\int_{\Omega}
\mathbf{1}_{E_{\lambda,\frac{\gamma}{p}}[f]}(\cdot,y)
\left|\cdot-y\right|^{\gamma-n}\,dy\right]^\frac{1}{p}
\right\|_{L^{r}_\omega(\Omega)}
\sim\left\|\,|\nabla f|\,\right\|_{L^{r}_\omega(\Omega)},
$$
where the positive equivalence constants are independent of $f$.
\end{theorem}

\begin{proof}
From $p\in[1,r]$ and \cite[p.\,86]{shyy2017},
it easily follows that
$[L^r_\omega(\mathbb{R}^n)]^\frac{1}{p}
=L^\frac{r}{p}_\omega(\mathbb{R}^n)$
and that both $L^r_\omega(\mathbb{R}^n)$ and
$[L^r_\omega(\mathbb{R}^n)]^\frac{1}{p}$
are ball Banach function spaces.
Notice that Lemma~\ref{ApProperty}(i) implies that
$\omega\in A_{\frac{r}{p}}(\mathbb{R}^n)
\subset A_r(\mathbb{R}^n)$.
By this and Lemma~\ref{ApProperty}(iii), we find that
the Hardy--Littlewood
maximal operator
$\mathcal{M}$
is bounded on $L^r_\omega(\mathbb{R}^n)$.
From the proof of \cite[Theorem~5.14]{zyy2023},
we infer that $\mathcal{M}$ is also bounded
on $([L^r_\omega(\mathbb{R}^n)]^\frac{1}{p})'$.
By \cite[Theorem~1.34]{rudin} and Lemma~\ref{ApProperty}(ii),
we conclude that both $L^r_\omega(\mathbb{R}^n)$
and $[L^r_\omega(\mathbb{R}^n)]'=
L^{r'}_{\omega^{1-r'}}(\mathbb{R}^n)$
have absolutely continuous norms.
Using these and Theorems~\ref{1931} and~\ref{1039},
we obtain the desired conclusions and hence
complete the proof of Theorem~\ref{2044}.
\end{proof}

\begin{remark}
\begin{enumerate}
\item[(i)]
To the best of our knowledge,
Theorem~\ref{2044} is completely new.
\item[(ii)]
Let $\omega\in A_1(\mathbb{R}^n)$.
In this case,
since $[L^1_\omega(\Omega)]'
=L^\infty_{\omega^{-1}}(\Omega)$
(see also \cite[p.\,9]{ins2019})
does not have an absolutely continuous
norm, it is still unclear whether or not
Theorem~\ref{2044} with $r=1$
holds true.
\end{enumerate}
\end{remark}

The following is a simple corollary of Theorem~\ref{2044}
when $\omega\equiv1$;
we omit the details here.

\begin{corollary}\label{Lq}
Let $\gamma\in\mathbb{R}\setminus\{0\}$, $q\in(1,\infty)$,
and $p\in[1,q]$.
Then $f\in\dot{W}^{1,q}(\Omega)$ if and only if
$f\in L_{{\mathrm{loc}}}^1(\Omega)$ and
$$
\sup_{\lambda\in(0,\infty)}\lambda
\left\|\left[\int_{\Omega}
\mathbf{1}_{E_{\lambda,\frac{\gamma}{p}}[f]}(\cdot,y)
\left|\cdot-y\right|^{\gamma-n}\,dy\right]^\frac{1}{p}
\right\|_{L^{q}(\Omega)}<\infty,
$$
where $E_{\lambda,\frac{\gamma}{p}}[f]$
for any $\lambda\in(0,\infty)$
is the same as in \eqref{Elambda};
moreover, for such $f$,
$$
\sup_{\lambda\in(0,\infty)}\lambda
\left\|\left[\int_{\Omega}
\mathbf{1}_{E_{\lambda,\frac{\gamma}{p}}[f]}(\cdot,y)
\left|\cdot-y\right|^{\gamma-n}\,dy\right]^\frac{1}{p}
\right\|_{L^{q}(\Omega)}
\sim\left\|\,|\nabla f|\,\right\|_{L^{q}(\Omega)},
$$
where the positive equivalence constants are independent of $f$.
\end{corollary}

\begin{remark}
Corollary~\ref{Lq} when $p=q\in(1,\infty)$
and $\Omega:=\mathbb{R}^n$ is
just \cite[Theorem~1.3]{bsvy.arxiv}
and when $p\in[1,q)$ or when $\Omega\subsetneqq
\mathbb{R}^n$ is new.
\end{remark}

We next turn to the case of Morrey spaces.
Let $0<r\leq\alpha<\infty$.
The \emph{Morrey space} $M_r^\alpha(\mathbb{R}^n)$,
introduced by Morrey \cite{m1938},
is defined to be the set of
all the $f\in\mathscr{M}(\mathbb{R}^n)$
having the finite quasi-norm
\begin{align*}
\|f\|_{M_r^\alpha(\mathbb{R}^n)}:=\sup_{B\in\mathbb{B}}
\left|B\right|^{\frac{1}{\alpha}-\frac{1}{r}}\|f\|_{L^r(B)},
\end{align*}
where $\mathbb{B}:=\{B(x,r):\ x\in\mathbb{R}^n,\,r\in(0,\infty)\}$.
Nowadays, Morrey spaces have found important
applications in the theory of
partial differential equations, harmonic analysis,
and potential theory
(see, for instance, the articles \cite{hms2017,hms2020,hss2018,
hs2017,tyy2019}
and the monographs \cite{sdh2020i,sdh20202,ysy2010}).
As was pointed out in \cite[p.\,87]{shyy2017},
$M_r^\alpha(\mathbb{R}^n)$ for any $1\leq r\leq\alpha<\infty$
is a ball Banach function space,
but is not a Banach function space
in the sense of Bennett and Sharpley \cite{bs1988}.
Since the Morrey space $M_r^\alpha(\mathbb{R}^n)$
does not have an absolutely continuous norm,
Theorems~\ref{1931} and~\ref{1039} seem to be inapplicable in
the Morrey setting.
However, we can use Theorem~\ref{2044} to further
obtain a characterization similar to Theorems~\ref{1931} and~\ref{1039}
for Morrey--Sobolev spaces.

\begin{theorem}\label{Morrey}
Let $\gamma\in\mathbb{R}\setminus\{0\}$,
$1<r\leq\alpha<\infty$, and $p\in[1,r]$.
Then $f\in\dot{W}^{1,M_r^\alpha}(\Omega)$
if and only if $f\in L_{{\mathrm{loc}}}^1(\Omega)$ and
\begin{align*}
\sup_{\lambda\in(0,\infty)}\lambda
\left\|\left[\int_{\Omega}
\mathbf{1}_{E_{\lambda,\frac{\gamma}{p}}[f]}(\cdot,y)
\left|\cdot-y\right|^{\gamma-n}\,dy\right]^\frac{1}{p}
\right\|_{M_r^\alpha(\Omega)}<\infty,
\end{align*}
where $E_{\lambda,\frac{\gamma}{p}}[f]$
for any $\lambda\in(0,\infty)$
is the same as in \eqref{Elambda};
moreover, for such $f$,
$$
\sup_{\lambda\in(0,\infty)}\lambda
\left\|\left[\int_{\Omega}
\mathbf{1}_{E_{\lambda,\frac{\gamma}{p}}[f]}(\cdot,y)
\left|\cdot-y\right|^{\gamma-n}\,dy\right]^\frac{1}{p}
\right\|_{M_r^\alpha(\Omega)}
\sim\left\|\,|\nabla f|\,\right\|_{M_r^\alpha(\Omega)},
$$
where the positive equivalence constants are independent of $f$.
\end{theorem}

\begin{proof}
Fix $\theta\in(1-\frac{r}{\alpha},1)$.
By \cite[Proposition~285]{sdh2020i},
we find that, for any $g\in\mathscr{M}(\mathbb{R}^n)$,
\begin{align}\label{2106}
\left\|g\right\|_{M_r^\alpha(\mathbb{R}^n)}\sim
\sup_{Q\subset\mathbb{R}^n}
|Q|^{\frac{1}{\alpha}-\frac{1}{r}}
\|g\|_{L^r_{[\mathcal{M}(\mathbf{1}_Q)]^\theta}(\mathbb{R}^n)},
\end{align}
where the supremum is taken over all cubes $Q\subset\mathbb{R}^n$
and where the the positive equivalence constants depend only on
$r$, $\alpha$, $n$, and $\theta$.
From \cite[Theorem~281]{sdh2020i} and $\theta\in(0,1)$,
we deduce that,
for any cube  $Q\subset\mathbb{R}^n$,
$[\mathcal{M}(\mathbf{1}_Q)]^\theta\in A_1(\mathbb{R}^n)$
and
$[\{\mathcal{M}(\mathbf{1}_Q)\}^\theta
]_{A_1(\mathbb{R}^n)}\lesssim1$,
where the implicit positive constant depends only on $\theta$.
By this, Remark~\ref{norm}(i),
\eqref{2106}, and Theorem~\ref{2044}, we conclude that,
for any $f\in L_{{\mathrm{loc}}}^1(\Omega)$,
\begin{align*}
\left\|\,|\nabla f|\,\right\|_{M_r^\alpha(\Omega)}
&=\left\|\widetilde{|\nabla f|}
\right\|_{M_r^\alpha(\mathbb{R}^n)}
\sim\sup_{Q\subset\mathbb{R}^n}
|Q|^{\frac{1}{\alpha}-\frac{1}{r}}
\left\|\widetilde{|\nabla f|}\right\|_
{L^r_{[\mathcal{M}(\mathbf{1}_Q)]^\theta}(\mathbb{R}^n)}\\
&=\sup_{Q\subset\mathbb{R}^n}
|Q|^{\frac{1}{\alpha}-\frac{1}{r}}
\left\|\,|\nabla f|\,\right\|_
{L^r_{[\mathcal{M}(\mathbf{1}_Q)]^\theta}(\Omega)}\\
&\sim\sup_{Q\subset\mathbb{R}^n}
|Q|^{\frac{1}{\alpha}-\frac{1}{r}}\\
&\quad\times\sup_{\lambda\in(0,\infty)}\lambda
\left\|\left[\int_{\Omega}
\mathbf{1}_{E_{\lambda,\frac{\gamma}{p}}[f]}(\cdot,y)
\left|\cdot-y\right|^{\gamma-n}\,dy\right]^\frac{1}{p}
\right\|_
{L^r_{[\mathcal{M}(\mathbf{1}_Q)]^\theta}(\Omega)}\\
&\sim\sup_{\lambda\in(0,\infty)}\lambda
\left\|\left[\int_{\Omega}
\mathbf{1}_{E_{\lambda,\frac{\gamma}{p}}[f]}(\cdot,y)
\left|\cdot-y\right|^{\gamma-n}\,dy\right]^\frac{1}{p}
\right\|_{M_r^\alpha(\Omega)},
\end{align*}
where $\widetilde{|\nabla f|}$
is defined the same as in \eqref{1448}
with $f$ replaced by $|\nabla f|$.
This finishes the proof of Theorem~\ref{Morrey}.
\end{proof}

\begin{remark}
To the best of our knowledge,
Theorem~\ref{Morrey} is completely new.
\end{remark}

\subsection{Besov--Bourgain--Morrey Spaces}
\label{BBMspace}

It is well known that Bourgain--Morrey spaces
(see \cite[Definition~1.1]{hnsh2022} for its definition),
whose special case was introduced by Bourgain \cite{bourgain},
play an essential role in the study related to both the
Strichartz estimate and the nonlinear Schr\"odinger equation.
As a generalization of the Bourgain--Morrey space,
the Besov--Bourgain--Morrey space was recently
introduced by Zhao et al. \cite{zstyy2022},
which is a bridge connecting Bourgain--Morrey spaces
with Fofana spaces;
see \cite[p.\,525]{f1989} for the definition
of the Fofana space.
We refer the reader to
\cite{hnsh2022,hly2023,m1,m2,zstyy2022}
for more studies on Bourgain--Morrey-type spaces
and to \cite{cf2019,f1988,ffk2015}
for more studies on Fofana spaces.

To give the definition of the
Besov--Bourgain--Morrey space,
we begin with recalling the following adjacent system
of dyadic cubes of $\mathbb{R}^n$,
which can be
found in \cite[Section~2.2]{lsu2012}.

\begin{lemma}\label{2115}
For any $\alpha\in\{0,\frac{1}{3},\frac{2}{3}\}^n$, let
$$
\mathcal{D}^\alpha:=\left\{
Q_{\nu,m}^{(\alpha)}:=2^\nu\left[m+(0,1]^n+(-1)^\nu\alpha\right]:\
\nu\in\mathbb{Z},\ m\in\mathbb{Z}^n\right\}.
$$
Then
\begin{enumerate}
\item[\textup{(i)}]
for any $Q,P\in\mathcal{D}^\alpha$ with
$\alpha\in\{0,\frac{1}{3},\frac{2}{3}\}^n$,
$Q\cap P\in\{\emptyset,Q,P\}$;
\item[\textup{(ii)}]
for any ball $B\subset\mathbb{R}^n$, there exist
$\alpha\in\{0,\frac{1}{3},\frac{2}{3}\}^n$ and $Q\in\mathcal{D}^\alpha$
such that $B\subset Q\subset CB$,
where the positive constant $C$ depends only on $n$.
\end{enumerate}
\end{lemma}

The following definitions of
the Besov--Bourgain--Morrey space,
the block, and the block space can be founded in
\cite[Definitions~1.2, 3.1, and 3.2]{zstyy2022}.

\begin{definition}\label{2209}
\begin{enumerate}
\item[(i)]
Let $0<q\leq p\leq r\leq\infty$, $\tau\in(0,\infty]$,
and $\mathcal{D}:=\{Q_{\nu,m}\}_{\nu\in\mathbb{Z},\,
m\in\mathbb{Z}^n}$ be a system of dyadic cubes of $\mathbb{R}^n$.
The \emph{Besov--Bourgain--Morrey space}
$\mathcal{M}\dot{B}^{p,\tau}_{q,r}(\mathbb{R}^n)$
is defined to be the set of all the $f\in L_{\mathrm{loc}}^q(\mathbb{R}^n)$
such that
$$
\left\|f\right\|_{\mathcal{M}\dot{B}^{p,\tau}_{q,r}(\mathbb{R}^n)}
:=\left\{\sum_{\nu\in\mathbb{Z}}
\left[\sum_{m\in\mathbb{Z}^n}\left\{|Q_{\nu,m}|^{\frac{1}{p}-\frac{1}{q}}
\left[\int_{Q_{\nu,m}}|f(y)|^q\,dy\right]^\frac{1}{q}\right\}^r
\right]^\frac{\tau}{r}\right\}^\frac{1}{\tau},
$$
with the usual modifications made when
$q=\infty$, $r=\infty$, or $\tau=\infty$,
is finite.
\item[(ii)]
Let $1\leq q\leq p\leq\infty$.
A function $b\in\mathscr{M}(\mathbb{R}^n)$ is
called a \emph{$(p',q')$-block} if there exists
a dyadic cube $Q\subset\mathbb{R}^n$ such that
$\mathrm{supp\,}(b)\subset Q$
and
$\|b\|_{L^{q'}(\mathbb{R}^n)}
\leq|Q|^{\frac{1}{q'}-\frac{1}{p'}}$.
\item[(iii)]
Let $1\leq q\leq p\leq r\leq\infty$
and $\tau\in[1,\infty]$. The \emph{block space}
$\mathcal{B}^{p',\tau'}_{q',r'}(\mathbb{R}^n)$
is defined to be the set of all the $f\in\mathscr{M}(\mathbb{R}^n)$
such that
$f=\sum_{\nu\in\mathbb{Z}}\sum_{m\in\mathbb{Z}^n}
\lambda_{\nu,m}b_{\nu,m}$
almost everywhere in $\mathbb{R}^n$,
where the sequence $\{\lambda_{\nu,m}\}_{\nu\in\mathbb{Z},\,m\in\mathbb{Z}^n}$
satisfies
$$
\left\|\{\lambda_{\nu,m}\}_{\nu\in\mathbb{Z},\,m\in\mathbb{Z}^n}
\right\|_{\ell_\nu^{\tau'}(\ell_m^{r'})}:=
\left[\sum_{\nu\in\mathbb{Z}}\left(\sum_{m\in\mathbb{Z}^n}
\left|\lambda_{\nu,m}\right|^{r'}
\right)^\frac{\tau'}{r'}\right]^\frac{1}{\tau'}<\infty
$$
and where $b_{\nu,m}$ is a $(p',q')$-block supported
in $Q_{\nu,m}$ for any $\nu\in\mathbb{Z}$ and $m\in\mathbb{Z}^n$.
Moreover, for any $f\in\mathcal{B}^{p',\tau'}_{q',r'}(\mathbb{R}^n)$,
$$
\left\|f\right\|_{\mathcal{B}^{p',\tau'}_{q',r'}(\mathbb{R}^n)}:=
\inf\left\{\left\|\{\lambda_{\nu,m}\}_{\nu\in\mathbb{Z},\,m\in\mathbb{Z}^n}
\right\|_{\ell_\nu^{\tau'}(\ell_m^{r'})}:\
f=\sum_{\nu\in\mathbb{Z}}\sum_{m\in\mathbb{Z}^n}
\lambda_{\nu,m}b_{\nu,m}\right\},
$$
where the infimum is taken over all the decompositions of $f$ as above.
\end{enumerate}
\end{definition}

The following lemma can be found in
\cite[Theorems~3.6 and~3.8]{zstyy2022}
which indicates that the dual space of
the Besov--Bourgain--Morrey space
is just the block space and that the dual space
of the block space is just the
Besov--Bourgain--Morrey space.

\begin{lemma}\label{duals}
\begin{enumerate}
\item[\rm(i)]
If $1<q<p<r\leq\infty$ and $\tau\in(1,\infty)$
or if $1<q\leq p\leq r\leq\infty$ and $\tau=\infty$,
then the dual space of $\mathcal{B}^{p',\tau'}_{q',r'}(\mathbb{R}^n)$
is $\mathcal{M}\dot{B}^{p,\tau}_{q,r}(\mathbb{R}^n)$
in the following sense:
\begin{enumerate}
\item[$\rm(i)_1$]
if $g\in\mathcal{M}\dot{B}^{p,\tau}_{q,r}(\mathbb{R}^n)$,
then the linear functional
\begin{align}\label{3.6}
\mathcal{J}_g:\ f\to\mathcal{J}_g(f)
:=\int_{\mathbb{R}^n}f(x)g(x)\,dx
\end{align}
is bounded on $\mathcal{B}^{p',\tau'}_{q',r'}(\mathbb{R}^n)$;
\item[$\rm(i)_2$]
conversely, any continuous linear functional on
$\mathcal{B}^{p',\tau'}_{q',r'}(\mathbb{R}^n)$ arises as
in \eqref{3.6} with a unique $g\in\mathcal{M}\dot{B}^{p,\tau}_{q,r}(\mathbb{R}^n)$;
\end{enumerate}
moreover, $\|g\|_{\mathcal{M}\dot{B}^{p,\tau}_{q,r}(\mathbb{R}^n)}
=\|\mathcal{J}_g\|_{[\mathcal{B}^{p',\tau'}_{q',r'}(\mathbb{R}^n)]^*}$.
\item[\rm(ii)]
If $1<q<p<r<\infty$ and $\tau\in(1,\infty)$,
then the dual space of $\mathcal{M}\dot{B}^{p,\tau}_{q,r}(\mathbb{R}^n)$
is $\mathcal{B}^{p',\tau'}_{q',r'}(\mathbb{R}^n)$ in the sense of (i)
with $\mathcal{M}\dot{B}^{p,\tau}_{q,r}(\mathbb{R}^n)$
and $\mathcal{B}^{p',\tau'}_{q',r'}(\mathbb{R}^n)$
replaced, respectively, by
$\mathcal{B}^{p',\tau'}_{q',r'}(\mathbb{R}^n)$ and
$\mathcal{M}\dot{B}^{p,\tau}_{q,r}(\mathbb{R}^n)$.
\end{enumerate}
\end{lemma}

Next, we prove that the Besov--Bourgain--Morrey space
is a ball Banach function space.

\begin{lemma}\label{BBFS}
If $1\leq q<p<r\leq\infty$ and $\tau\in[1,\infty)$
or if $1\leq q\leq p\leq r\leq\tau=\infty$,
then $\mathcal{M}\dot{B}^{p,\tau}_{q,r}(\mathbb{R}^n)$
is a ball Banach function space.
\end{lemma}

\begin{proof}
We only consider the case that $r,\tau\in[1,\infty)$
because the proofs of the cases that
$r=\infty$ or $\tau=\infty$ are similar and hence we omit the details.
By the triangle inequalities of $L^q(\mathbb{R}^n)$,
$\ell^r$, and $\ell^\tau$,
we find that $\|\cdot\|_{\mathcal{M}
\dot{B}^{p,\tau}_{q,r}(\mathbb{R}^n)}$ satisfies
Definition~\ref{1659}(v).
From the definition of $\mathcal{M}\dot{B}^{p,\tau}_{q,r}(\mathbb{R}^n)$,
we infer that, if $f\in\mathscr{M}(\mathbb{R}^n)$
and $\|f\|_{\mathcal{M}\dot{B}^{p,\tau}_{q,r}(\mathbb{R}^n)}=0$,
then, for any $\nu\in\mathbb{Z}$, $m\in\mathbb{Z}^n$,
and almost every $x\in Q_{\nu,m}$, $f(x)=0$,
which, combined with the fact that
$\bigcup_{m\in\mathbb{Z}^n}Q_{\nu,m}=\mathbb{R}^n$
for any $\nu\in\mathbb{Z}$,
further implies that $f=0$ almost everywhere in $\mathbb{R}^n$.
Thus, $\|\cdot\|_{\mathcal{M}\dot{B}^{p,\tau}_{q,r}(\mathbb{R}^n)}$ satisfies
Definition~\ref{1659}(i).
By the definition of $\|\cdot\|_{\mathcal{M}
\dot{B}^{p,\tau}_{q,r}(\mathbb{R}^n)}$,
we are easy to show that $\|\cdot\|_{\mathcal{M}
\dot{B}^{p,\tau}_{q,r}(\mathbb{R}^n)}$
satisfies Definition~\ref{1659}(ii).
From the monotone convergence theorem,
it is easy to deduce that $\mathcal{M}\dot{B}^{p,\tau}_{q,r}(\mathbb{R}^n)$
satisfies Definition~\ref{1659}(iii).
This, together with both Definition~\ref{1659}(ii)
and Remark~\ref{1052}(iv), further implies that
$\mathcal{M}\dot{B}^{p,\tau}_{q,r}(\mathbb{R}^n)$ is complete.
Thus, $\mathcal{M}\dot{B}^{p,\tau}_{q,r}(\mathbb{R}^n)$ is a Banach space.
By the proof of \cite[Proposition~2.15]{zstyy2022},
we find that the characteristic function
of any dyadic cube belongs to $\mathcal{M}
\dot{B}^{p,\tau}_{q,r}(\mathbb{R}^n)$.
This, combined with both (ii) and (v) of Definition~\ref{1659},
further implies that, for any ball $B$, $\mathbf{1}_B\in
\mathcal{M}\dot{B}^{p,\tau}_{q,r}(\mathbb{R}^n)$.
Thus, $\mathcal{M}\dot{B}^{p,\tau}_{q,r}(\mathbb{R}^n)$ satisfies
Definition~\ref{1659}(iv).
In addition, from Lemma~\ref{2115},
we infer that,
for any ball $B\subset\mathbb{R}^n$, there exist
$\alpha\in\{0,\frac{1}{3},\frac{2}{3}\}^n$
and $Q\in\mathcal{D}^\alpha$
such that $B\subset Q$ and $|B|\sim|Q|$,
which, together with the H\"older inequality,
further implies that, for any $f\in\mathcal{M}
\dot{B}^{p,\tau}_{q,r}(\mathbb{R}^n)$,
\begin{align*}
\int_B|f(x)|\,dx&\leq\int_Q|f(x)|\,dx
\leq|Q|^\frac{1}{q'}\|f\|_{L^q(Q)}
=|Q|^{\frac{1}{p'}}|Q|^{\frac{1}{p}-\frac{1}{q}}\|f\|_{L^q(Q)}\\
&\leq|Q|^{\frac{1}{p'}}
\left\{\sum_{\nu\in\mathbb{Z}}
\left[\sum_{m\in\mathbb{Z}^n}\left\{\left|Q_{\nu,m}^{(\alpha)}
\right|^{\frac{1}{p}-\frac{1}{q}}
\left[\int_{Q_{\nu,m}^{(\alpha)}}|f(y)|^q\,dy\right]^\frac{1}{q}\right\}^r
\right]^\frac{\tau}{r}\right\}^\frac{1}{\tau}\\
&\sim|B|^{\frac{1}{p'}}
\left\{\sum_{\nu\in\mathbb{Z}}
\left[\sum_{m\in\mathbb{Z}^n}\left\{\left|Q_{\nu,m}^{(\alpha)}
\right|^{\frac{1}{p}-\frac{1}{q}}
\left[\int_{Q_{\nu,m}^{(\alpha)}}|f(y)|^q\,dy\right]^\frac{1}{q}\right\}^r
\right]^\frac{\tau}{r}\right\}^\frac{1}{\tau}\\
&\sim|B|^{\frac{1}{p'}}\|f\|_{\mathcal{M}
\dot{B}^{p,\tau}_{q,r}(\mathbb{R}^n)},
\end{align*}
where, in the last step, we used the fact that the space
$\mathcal{M}\dot{B}^{p,\tau}_{q,r}(\mathbb{R}^n)$ is independent of
the choice of the system of dyadic cubes
which can be deduced from \cite[Theorem~2.9]{zstyy2022}.
Thus, $\mathcal{M}\dot{B}^{p,\tau}_{q,r}(\mathbb{R}^n)$ satisfies
Definition~\ref{1659}(vi).
This finishes the proof of Lemma~\ref{BBFS}.
\end{proof}

Now,
we establish the boundedness of the Hardy--Littlewood
maximal operator on block spaces
in Definition~\ref{2209}(iii).

\begin{proposition}\label{2243}
If $1<q<p<r<\infty$ and $\tau\in(1,\infty)$,
then the Hardy--Littlewood maximal operator $\mathcal{M}$
is bounded on the block space $\mathcal{B}^{p',\tau'}_{q',r'}(\mathbb{R}^n)$.
\end{proposition}

\begin{proof}
Since $1<q<p<r<\infty$ and $\tau\in(1,\infty)$,
it follows that
we can choose one
$s\in(1,\min\{q',\,p',\,r',\,\tau'\})$.
Using Lemma~\ref{duals}(ii),
we obtain
\begin{align*}
\left[\mathcal{M}\dot{B}^{p,\tau}_{q,r}(\mathbb{R}^n)\right]'
=\mathcal{B}^{p',\tau'}_{q',r'}(\mathbb{R}^n)
\ \text{and}\
\left[\mathcal{M}\dot{B}^{(\frac{p'}{s})',(\frac{\tau'}{s})'}_{
(\frac{q'}{s})',(\frac{r'}{s})'}(\mathbb{R}^n)\right]'
=\mathcal{B}^{\frac{p'}{s},\frac{\tau'}{s}
}_{\frac{q'}{s},\frac{r'}{s}}(\mathbb{R}^n),
\end{align*}
which, together with both Remark~\ref{dual}(i)
and Lemma~\ref{BBFS},
further implies that both
$\mathcal{B}^{p',\tau'}_{q',r'}(\mathbb{R}^n)$ and
$[\mathcal{B}^{p',\tau'}_{q',r'}
(\mathbb{R}^n)]^\frac{1}{s}$
are ball Banach function spaces.
Moreover,
by Lemma~\ref{duals}(i),
we find that
\begin{align*}
\left(\left[\mathcal{B}^{p',\tau'}_{q',r'}
(\mathbb{R}^n)\right]^\frac{1}{s}\right)'=
\mathcal{M}\dot{B}^{(\frac{p'}{s})',
(\frac{\tau'}{s})'}_{(\frac{q'}{s})',(\frac{r'}{s})'}(\mathbb{R}^n),
\end{align*}
which, combined with \cite[Corollary~4.7]{zstyy2022},
further implies that $\mathcal{M}$ is bounded on
$([\mathcal{B}^{p',\tau'}_{q',r'}
(\mathbb{R}^n)]^\frac{1}{s})'$.
From these and Lemma~\ref{2005} with
both $X:=\mathcal{B}^{p',\tau'}_{q',r'}(\mathbb{R}^n)$
and $p:=s$,
we deduce that $\mathcal{M}$
is bounded on $\mathcal{B}^{p',\tau'}_{q',r'}(\mathbb{R}^n)$,
which completes the proof of Proposition~\ref{2243}.
\end{proof}

Using Theorems~\ref{1931} and~\ref{1039}
and Proposition~\ref{2243}, we obtain the following conclusion.

\begin{theorem}\label{BBM}
If $\gamma\in\mathbb{R}\setminus\{0\}$,
$1<q<p<r<\infty$, $\tau\in(1,\infty)$,
and $s\in[1,\min\{q,\,p,\,r,\,\tau\})$,
then $f\in\dot{W}^{1,\mathcal{M}\dot{B}^{p,\tau}_{q,r}}(\Omega)$
if and only if
$f\in L_{{\mathrm{loc}}}^1(\Omega)$ and
$$
\sup_{\lambda\in(0,\infty)}\lambda
\left\|\left[\int_{\Omega}
\mathbf{1}_{E_{\lambda,\frac{\gamma}{s}}[f]}(\cdot,y)
\left|\cdot-y\right|^{\gamma-n}\,dy\right]^\frac{1}{s}
\right\|_{\mathcal{M}\dot{B}^{p,\tau}_{q,r}(\Omega)}<\infty,
$$
where $E_{\lambda,\frac{\gamma}{s}}[f]$
for any $\lambda\in(0,\infty)$
is the same as in \eqref{Elambda};
moreover, for such $f$,
$$
\sup_{\lambda\in(0,\infty)}\lambda
\left\|\left[\int_{\Omega}
\mathbf{1}_{E_{\lambda,\frac{\gamma}{s}}[f]}(\cdot,y)
\left|\cdot-y\right|^{\gamma-n}\,dy\right]^\frac{1}{s}
\right\|_{\mathcal{M}\dot{B}^{p,\tau}_{q,r}(\Omega)}
\sim\left\|\,|\nabla f|\,\right\|_{\mathcal{M}
\dot{B}^{p,\tau}_{q,r}(\Omega)},
$$
where the positive equivalence constants are independent of $f$.
\end{theorem}

\begin{proof}
By Lemma~\ref{BBFS},
we find that both $\mathcal{M}\dot{B}^{p,\tau}_{q,r}(\mathbb{R}^n)$
and $[\mathcal{M}\dot{B}^{p,\tau}_{q,r}(\mathbb{R}^n)]^\frac{1}{s}$
are ball Banach function spaces.
From Lemma~\ref{duals},
we infer that $\mathcal{M}\dot{B}^{p,\tau}_{q,r}(\mathbb{R}^n)$
is reflexive,
which together with Lemma~\ref{reflexive},
further implies that both $\mathcal{M}\dot{B}^{p,\tau}_{q,r}(\mathbb{R}^n)$
and $[\mathcal{M}\dot{B}^{p,\tau}_{q,r}(\mathbb{R}^n)]'$
have absolutely continuous norms.
Using \cite[Corollary~4.7]{zstyy2022}, we conclude that
the Hardy--Littlewood maximal operator $\mathcal{M}$
is bounded on $\mathcal{M}\dot{B}^{p,\tau}_{q,r}(\mathbb{R}^n)$.
Moreover, from $s\in[1,\min\{q,\,p,\,r,\,\tau\})$
and Proposition~\ref{2243},
we deduce that $\mathcal{M}$
is bounded on
$$
\left(\left[\mathcal{M}\dot{B}^{p,\tau}_{q,r}
(\mathbb{R}^n)\right]^\frac{1}{s}\right)'
=\mathcal{B}^{(\frac{p}{s})',(\frac{\tau}{s})'}_{
(\frac{q}{s})',(\frac{r}{s})'}(\mathbb{R}^n).
$$
By these and Theorems~\ref{1931} and~\ref{1039},
we then obtain the desired conclusion and hence complete
the proof of Theorem~\ref{BBM}.
\end{proof}

\begin{remark}
\begin{enumerate}
\item[\textup{(i)}]
To the best of our knowledge,
Theorem~\ref{BBM} is completely new.
\item[\textup{(ii)}]
Let $\tau\in[1,\infty)$ and $1=q<p<r<\infty$ or let
$\tau=1\leq q<p<r<\infty$.
In both cases, since the associate space
of $\mathcal{M}\dot{B}^{p,\tau}_{q,r}(\Omega)$ is unknown,
it is still unclear whether or not Theorem~\ref{BBM} when $s=1$
in both cases holds true.
\end{enumerate}
\end{remark}

\subsection{Local and Global Generalized Herz Spaces}
\label{Herz}

We begin with recalling several basic concepts
related to local and global generalized Herz spaces.
A nonnegative function $\omega$ on $\mathbb{R}_+$
is said to be \emph{almost increasing}
(resp. \emph{almost decreasing})
if there exists a constant $C\in[1,\infty)$ such that,
for any $0<t\leq\tau<\infty$ (resp. $0<\tau\leq t<\infty$),
$$
\omega(t)\leq C\omega(\tau).
$$
The \emph{function class} $M(\mathbb{R}_+)$ is defined to be
the set of all the positive functions $\omega$ on $\mathbb{R}_+$
such that, for any $0<\delta<N<\infty$,
$$
0<\inf_{t\in(\delta,N)}\omega(t)
\leq\sup_{t\in(\delta,N)}\omega(t)<\infty
$$
and that there exist four constants
$\alpha_0,\beta_0,\alpha_\infty,
\beta_\infty\in\mathbb{R}$ such that
\begin{enumerate}
\item[(i)]
for any $t\in(0,1]$,
$t^{-\alpha_0}\omega(t)$ is almost increasing
and $t^{-\beta_0}\omega(t)$ is almost decreasing;
\item[(ii)]
for any $t\in[1,\infty)$,
$t^{-\alpha_\infty}\omega(t)$ is almost increasing
and $t^{-\beta_\infty}\omega(t)$ is almost decreasing.
\end{enumerate}
The \emph{Matuszewska--Orlicz indices} $m_0(\omega)$, $M_0(\omega)$,
$m_\infty(\omega)$, and $M_\infty(\omega)$
of a positive function $\omega$ on $\mathbb{R}_+$
are defined, respectively, by setting
\begin{align*}
m_0(\omega):=\sup_{t\in(0,1)}\frac{\ln\,[\limsup\limits_{h\to0^+}
\frac{\omega(ht)}{\omega(h)}]}{\ln t},\
M_0(\omega):=\inf_{t\in(0,1)}\frac{
\ln\,[\liminf\limits_{h\to0^+}
\frac{\omega(ht)}{\omega(h)}]}{\ln t},
\end{align*}
\begin{align*}
m_\infty(\omega):=\sup_{t\in(1,\infty)}\frac{
\ln\,[\liminf\limits_{h\to\infty}
\frac{\omega(ht)}{\omega(h)}]}{\ln t},
\end{align*}
and
\begin{align*}
M_\infty(\omega):=\inf_{t\in(1,\infty)}
\frac{\ln\,[\limsup\limits_{h\to\infty}
\frac{\omega(ht)}{\omega(h)}]}{\ln t}.
\end{align*}
Next, we recall the concepts of both the local generalized Herz space
and the global generalized Herz space,
which can be found in \cite[Definitions~1.2.1
and  1.7.1]{lyh2320}
(see also \cite[Definition~2.2]{rs2020}).
For any $\xi\in\mathbb{R}^n$,
$L_{\mathrm{loc}}^p(\mathbb{R}^n\setminus\{\xi\})$ denotes the set
of all the locally $p$-integrable functions
on $\mathbb{R}^n\setminus\{\xi\}$.
For any $\xi\in\mathbb{R}^n$ and $k\in\mathbb{Z}$,
let
$R_{\xi,k}:=B(\xi,2^k)\setminus B(\xi,2^{k-1})$.

\begin{definition}\label{4.8}
Let $p,q\in(0,\infty]$ and $\omega\in M(\mathbb{R}_+)$.
\begin{enumerate}
\item[(i)]
The \emph{local generalized Herz space}
$\dot{\mathcal{K}}^{p,q}_{\omega,\xi}(\mathbb{R}^n)$,
with a given $\xi\in\mathbb{R}^n$,
is defined to be the set of all the $f\in L_{\mathrm{loc}}^p
(\mathbb{R}^n\setminus\{\xi\})$ having the following finite quasi-norm
\begin{align*}
\|f\|_{\dot{\mathcal{K}}^{p,q}_{\omega,\xi}(\mathbb{R}^n)}
:=\left\{\sum_{k\in\mathbb{Z}}
\left[\omega(2^k)\right]^q
\left\|f\right\|_{L^p(R_{\xi,k})}^q
\right\}^\frac{1}{q}.
\end{align*}
\item[(ii)]
The \emph{global generalized Herz space}
$\dot{\mathcal{K}}^{p,q}_{\omega}(\mathbb{R}^n)$
is defined to be the set of all the $f\in L_{\mathrm{loc}}^p
(\mathbb{R}^n)$ having the following finite quasi-norm
\begin{align*}
\|f\|_{\dot{\mathcal{K}}^{p,q}_{\omega}(\mathbb{R}^n)}
:=\sup_{\xi\in\mathbb{R}^n}
\|f\|_{\dot{\mathcal{K}}^{p,q}_{\omega,\xi}(\mathbb{R}^n)}.
\end{align*}
\end{enumerate}
\end{definition}
Recall that the classical Herz space was originally
introduced by Herz \cite{herz} to
study the Bernstein theorem on absolutely
convergent Fourier transforms.
The local generalized Herz space
and the global generalized Herz space
in Definition~\ref{4.8}, introduced by
Rafeiro and Samko \cite{rs2020},
are the generalization of classical homogeneous
Herz spaces and connect with generalized
Morrey type spaces.
We refer the reader to \cite{gly1998,hy1999,hwyy2023,ly1996,
lyh2320,rs2020,zyz2022}
for more studies on Herz spaces.
As was pointed out in \cite[Theorem~1.2.46]{lyh2320},
for any $\xi\in\mathbb{R}^n$,
when $p,q\in[1,\infty]$ and $\omega\in M(\mathbb{R}_+)$
satisfy $-\frac{n}{p}<m_0(\omega)\leq M_0(\omega)<n-\frac{n}{p}$,
the local generalized Herz space
$\dot{\mathcal{K}}^{p,q}_{\omega,\xi}(\mathbb{R}^n)$
is a ball Banach function space.
Moreover, as was pointed out in \cite[Theorem~1.2.48]{lyh2320},
when $p,q\in[1,\infty]$ and
$\omega\in M(\mathbb{R}_+)$ satisfy both
$m_0(\omega)\in(-\frac{n}{p},\infty)$
and $M_\infty(\omega)\in(-\infty,0)$,
the global generalized Herz space
$\dot{\mathcal{K}}^{p,q}_{\omega}(\mathbb{R}^n)$ is
a ball Banach function space.

Using Theorems~\ref{1931} and~\ref{1039},
we obtain the following conclusion.

\begin{theorem}\label{localHerz}
Let $\gamma\in\mathbb{R}\setminus\{0\}$, $\xi\in\mathbb{R}^n$,
$p,q\in(1,\infty)$, $s\in[1,\min\{p,\,q\}]$, and
$\omega\in M(\mathbb{R}_+)$
satisfy
\begin{align}\label{858}
-\frac{n}{p}<m_0(\omega)\leq
M_0(\omega)<n\left(\frac{1}{s}-\frac{1}{p}\right)
\end{align}
and
\begin{align}\label{859}
-\frac{n}{p}<m_\infty(\omega)\leq M_\infty(\omega)
<n\left(\frac{1}{s}-\frac{1}{p}\right).
\end{align}
Then $f\in\dot{W}^{1,\dot{\mathcal{K}}^{p,q}_{\omega,
\xi}}(\Omega)$ if and only if
$f\in L_{{\mathrm{loc}}}^1(\Omega)$ and
$$
\sup_{\lambda\in(0,\infty)}\lambda
\left\|\left[\int_{\Omega}
\mathbf{1}_{E_{\lambda,\frac{\gamma}{p}}[f]}(\cdot,y)
\left|\cdot-y\right|^{\gamma-n}\,dy\right]^\frac{1}{p}
\right\|_{\dot{\mathcal{K}}^{p,q}_{\omega,\xi}
(\Omega)}<\infty,
$$
where $E_{\lambda,\frac{\gamma}{p}}[f]$
for any $\lambda\in(0,\infty)$
is the same as in \eqref{Elambda};
moreover, for such $f$,
\begin{align}\label{852}
\sup_{\lambda\in(0,\infty)}\lambda
\left\|\left[\int_{\Omega}
\mathbf{1}_{E_{\lambda,\frac{\gamma}{p}}[f]}(\cdot,y)
\left|\cdot-y\right|^{\gamma-n}\,dy\right]^\frac{1}{p}
\right\|_{\dot{\mathcal{K}}^{p,q}_{\omega,\xi}(\Omega)}
\sim\left\|\,|\nabla f|\,\right\|_{\dot{\mathcal{K}}^{
p,q}_{\omega,\xi}(\Omega)},
\end{align}
where the positive equivalence constants are independent of both $f$
and $\xi$.
\end{theorem}

\begin{proof}
By \cite[Lemma~1.1.6]{lyh2320} and $\omega\in M(\mathbb{R}_+)$,
we find that
\begin{align}\label{2149}
\omega^{-1},\omega^s,\omega^{-s}\in M(\mathbb{R}_+).
\end{align}
From \cite[Lemma~1.3.1]{lyh2320},
we infer that
$[\dot{\mathcal{K}}^{p,q}_{\omega,\xi}(\mathbb{R}^n)
]^\frac{1}{s}=\dot{\mathcal{K}}^{\frac{p}{s},
\frac{q}{s}}_{\omega^s,\xi}(\mathbb{R}^n)$.
Using this, $p,q\in(1,\infty)$, $s\in[1,\min\{p,\,q\}]$, \eqref{858},
\eqref{2149} [here, we need $\omega^s\in M(\mathbb{R}_+)$],
and \cite[Theorem~1.2.46]{lyh2320},
we find that both $\dot{\mathcal{K}}^{
p,q}_{\omega,\xi}(\mathbb{R}^n)$ and
$[\dot{\mathcal{K}}^{p,q}_{\omega,\xi
}(\mathbb{R}^n)]^\frac{1}{s}$
are ball Banach function spaces.
From $s\in[1,\min\{p,\,q\}]$, \eqref{858},
and \cite[Theorem~1.7.9]{lyh2320},
we deduce that
$$
\left(\left[\dot{\mathcal{K}}^{p,q}_{
\omega,\xi}(\mathbb{R}^n)
\right]^\frac{1}{s}\right)'
=\dot{\mathcal{K}}^{(\frac{p}{s})',(\frac{q}{s})'}_{
\omega^{-s},\xi}(\mathbb{R}^n),
$$
which, together with \eqref{859}, \eqref{2149}
[here, we need $\omega^{-s}\in M(\mathbb{R}_+)$], and
\cite[Lemma~1.1.6 and Corollary~1.5.4]{lyh2320},
further implies that the
Hardy--Littlewood maximal operator $\mathcal{M}$ is bounded
on both $\dot{\mathcal{K}}^{p,q}_{\omega,\xi}(\mathbb{R}^n)$ and
$([\dot{\mathcal{K}}^{p,q}_{
\omega,\xi}(\mathbb{R}^n)]^\frac{1}{s})'$
with both $\|\mathcal{M}\|_{
\dot{\mathcal{K}}^{p,q}_{\omega,\xi}(\mathbb{R}^n)
\to\dot{\mathcal{K}}^{p,q}_{\omega,\xi}(\mathbb{R}^n)}$
and $\|\mathcal{M}\|_{([\dot{\mathcal{K}}^{p,q}_{
\omega,\xi}(\mathbb{R}^n)]^\frac{1}{s})'
\to([\dot{\mathcal{K}}^{p,q}_{
\omega,\xi}(\mathbb{R}^n)]^\frac{1}{s})'}$
independent of $\xi$.
On the other hand, by $p,q\in(1,\infty)$,
\eqref{2149}
[here, we need $\omega^{-1}\in M(\mathbb{R}_+)$],
and \cite[Theorem~1.4.1]{lyh2320},
we find that both
$\dot{\mathcal{K}}^{p,q}_{\omega,\xi}(\mathbb{R}^n)$
and $[\dot{\mathcal{K}}^{p,q}_{\omega,\xi}(\mathbb{R}^n)]'
=\dot{\mathcal{K}}^{p',q'}_{\omega^{-1},\xi}(\mathbb{R}^n)$ have
absolutely continuous norms.
By these and Theorems~\ref{1931} and~\ref{1039},
we obtain the desired conclusions and hence
complete the proof of Theorem~\ref{localHerz}.
\end{proof}

\begin{remark}
\begin{enumerate}
\item[(i)]
To the best of our knowledge,
Theorem~\ref{localHerz} is completely new.
\item[(ii)]
Let $p,q\in[1,\infty)$ satisfy $p=1$ or $q=1$
and let $\gamma$, $s$, and $\omega$ be
the same as in Theorem~\ref{localHerz}.
In this case,
since $[\dot{\mathcal{K}}^{p,q}_{\omega,\xi}
(\Omega)]'$
does not have an absolutely continuous norm,
it is still unclear whether or not
Theorem~\ref{localHerz} with $p=1$ or $q=1$ holds true.
\end{enumerate}
\end{remark}

Since the global generalized Herz space does not
have an absolutely continuous norm,
Theorems~\ref{1931} and~\ref{1039} seem to be inapplicable
in this setting. However,
using Theorem~\ref{localHerz},
we obtain the following characterization
similar to Theorems~\ref{1931} and~\ref{1039} for
global generalized Herz--Sobolev spaces.

\begin{theorem}\label{globalHerz}
If $\gamma$, $p$, $q$, $s$, and $\omega$
are the same as in Theorem~\ref{localHerz},
then $f\in\dot{W}^{1,\dot{\mathcal{K}}^{p,q}_{\omega}}(\Omega)$
if and only if
$f\in L_{{\mathrm{loc}}}^1(\Omega)$ and
$$
\sup_{\lambda\in(0,\infty)}\lambda
\left\|\left[\int_{\Omega}
\mathbf{1}_{E_{\lambda,\frac{\gamma}{p}}[f]}(\cdot,y)
\left|\cdot-y\right|^{\gamma-n}\,dy\right]^\frac{1}{p}
\right\|_{\dot{\mathcal{K}}^{p,q}_{\omega}
(\Omega)}<\infty,
$$
where $E_{\lambda,\frac{\gamma}{p}}[f]$
for any $\lambda\in(0,\infty)$
is the same as in \eqref{Elambda};
moreover, for such $f$,
$$
\sup_{\lambda\in(0,\infty)}\lambda
\left\|\left[\int_{\Omega}
\mathbf{1}_{E_{\lambda,\frac{\gamma}{p}}[f]}(\cdot,y)
\left|\cdot-y\right|^{\gamma-n}\,dy\right]^\frac{1}{p}
\right\|_{\dot{\mathcal{K}}^{p,q}_{\omega}(\Omega)}
\sim\left\|\,|\nabla f|\,\right\|_{\dot{\mathcal{K}}^{
p,q}_{\omega}(\Omega)},
$$
where the positive equivalence constants are independent of $f$.
\end{theorem}

\begin{proof}
By Remark~\ref{norm}(ii),
Definition~\ref{4.8}(ii), Theorem~\ref{localHerz},
and the fact that the positive equivalence constants
in \eqref{852} are independent of $\xi\in\mathbb{R}^n$,
we conclude that,
for any $f\in L_{{\mathrm{loc}}}^1(\Omega)$,
\begin{align*}
\left\|\,|\nabla f|\,\right\|_{\dot{\mathcal{K}}^{p,q}_{\omega}
(\Omega)}
&=\left\|\widetilde{|\nabla f|}
\right\|_{\dot{\mathcal{K}}^{p,q}_{\omega}
(\mathbb{R}^n)}\\
&=\sup_{\xi\in\mathbb{R}^n}
\left\|\widetilde{|\nabla f|}\right\|_{
\dot{\mathcal{K}}^{p,q}_{\omega,
\xi}(\mathbb{R}^n)}
=\sup_{\xi\in\mathbb{R}^n}
\left\|\,\left|\nabla f\right|\,\right\|_{
\dot{\mathcal{K}}^{p,q}_{\omega,
\xi}(\Omega)}\\
&\sim\sup_{\xi\in\mathbb{R}^n}
\sup_{\lambda\in(0,\infty)}\lambda
\left\|\left[\int_{\Omega}
\mathbf{1}_{E_{\lambda,\frac{\gamma}{p}}[f]}(\cdot,y)
\left|\cdot-y\right|^{\gamma-n}\,dy\right]^\frac{1}{p}
\right\|_{\dot{\mathcal{K}}^{p,q}_{\omega,
\xi}(\Omega)}\\
&=\sup_{\lambda\in(0,\infty)}\lambda
\left\|\left[\int_{\Omega}
\mathbf{1}_{E_{\lambda,\frac{\gamma}{p}}[f]}(\cdot,y)
\left|\cdot-y\right|^{\gamma-n}\,dy\right]^\frac{1}{p}
\right\|_{\dot{\mathcal{K}}^{p,q}_{\omega}
(\Omega)},
\end{align*}
where $\widetilde{|\nabla f|}$
is defined the same as in \eqref{1448}
with $f$ replaced by $|\nabla f|$.
This finishes the proof of Theorem~\ref{globalHerz}.
\end{proof}

\begin{remark}
To the best of our knowledge,
Theorem~\ref{globalHerz} is completely new.
\end{remark}

\subsection{Mixed-Norm Lebesgue Spaces}\label{5.2}

Let $\vec{r}:=(r_1,\ldots,r_n)
\in(0,\infty]^n$ and
$r_-:=\min\{r_1, \ldots , r_n\}$.
The \emph{mixed-norm Lebesgue
space $L^{\vec{r}}(\mathbb{R}^n)$} is defined to be the
set of all the $f\in\mathscr{M}(\mathbb{R}^n)$
having the following finite quasi-norm
\begin{equation*}
\|f\|_{L^{\vec{r}}(\mathbb{R}^n)}:=\left\{\int_{\mathbb{R}}
\cdots\left[\int_{\mathbb{R}}\left|f(x_1,\ldots,
x_n)\right|^{r_1}\,dx_1\right]^{\frac{r_2}{r_1}}
\cdots\,dx_n\right\}^{\frac{1}{r_n}}
\end{equation*}
with the usual modifications made when $r_i=
\infty$ for some $i\in\{1,\ldots,n\}$.
The study of mixed-norm Lebesgue spaces
can be traced back to H\"ormander \cite{h1960}
and Benedek and Panzone \cite{bp1961}.
For more studies on mixed-norm Lebesgue spaces,
we refer the reader
to \cite{cgn2017,cgn2019,gjn2017,gn2016,
hlyy2019,hlyy2019b,hy2021,n2019,noss2021,zyz2022}.
When $\vec{r}\in(0,\infty)^n$,
from the definition of $L^{\vec{r}}(\mathbb{R}^n)$,
we easily deduce that
$L^{\vec{r}}(\mathbb{R}^n)$
is a ball quasi-Banach function space.
But $L^{\vec{r}}(\mathbb{R}^n)$
may not be a quasi-Banach function space
in the sense of Bennett and Sharpley \cite{bs1988}
(see, for instance, \cite[Remark 7.20]{zwyy2021}).
When $X:=L^{\vec{r}}$, we
simply write $\dot{W}^{1,\vec{r}}(\Omega)
:=\dot{W}^{1,X}(\Omega)$.

Using Theorems~\ref{1931} and~\ref{1039},
we obtain the following conclusion.

\begin{theorem}\label{2116}
Let $\gamma\in\mathbb{R}\setminus\{0\}$,
$\vec{r}:=(r_1,\ldots,r_n)\in(1,\infty)^n$,
and $p\in[1,r_-)$.
Then $f\in\dot{W}^{1,\vec{r}}(\Omega)$ if and only if
$f\in L_{{\mathrm{loc}}}^1(\Omega)$ and
$$
\sup_{\lambda\in(0,\infty)}\lambda
\left\|\left[\int_{\Omega}
\mathbf{1}_{E_{\lambda,\frac{\gamma}{p}}[f]}(\cdot,y)
\left|\cdot-y\right|^{\gamma-n}\,dy\right]^\frac{1}{p}
\right\|_{L^{\vec{r}}(\Omega)}<\infty,
$$
where $E_{\lambda,\frac{\gamma}{p}}[f]$
for any $\lambda\in(0,\infty)$
is the same as in \eqref{Elambda};
moreover, for such $f$,
$$
\sup_{\lambda\in(0,\infty)}\lambda
\left\|\left[\int_{\Omega}
\mathbf{1}_{E_{\lambda,\frac{\gamma}{p}}[f]}(\cdot,y)
\left|\cdot-y\right|^{\gamma-n}\,dy\right]^\frac{1}{p}
\right\|_{L^{\vec{r}}(\Omega)}
\sim\left\|\,|\nabla f|\,\right\|_{L^{\vec{r}}(\Omega)},
$$
where the positive equivalence constants are independent of $f$.
\end{theorem}

\begin{proof}
Notice that $[L^{\vec{r}}(\mathbb{R}^n)]^{\frac{1}{p}}
=L^{\frac{\vec{r}}{p}}(\mathbb{R}^n)$.
By this, $p\in[1,r_-)$,
and \cite[p.\,304, Theorem 1.b)]{bp1961}
(see also \cite[Remark~2.8(ii)]{hlyy2019}),
we find that both $L^{\vec{r}}(\mathbb{R}^n)$ and
$[L^{\vec{r}}(\mathbb{R}^n)]^{\frac{1}{p}}$
are ball Banach function spaces.
From \cite[p.\,304, Theorem~1.a)]{bp1961}
and $1<\frac{r_-}{p}\leq r_+<\infty$,
we infer that
$([L^{\vec{r}}(\mathbb{R}^n)]^{\frac{1}{p}})'
=L^{(\frac{\vec{r}}{p})'}(\mathbb{R}^n)$,
where $(\frac{\vec{r}}{p})'
:=((\frac{r_1}{p})',\ldots,(\frac{r_n}{p})')$,
which, combined with \cite[Lemma~3.5]{hlyy2019},
further implies that the
Hardy--Littlewood maximal operator $\mathcal{M}$ is bounded
on both $L^{\vec{r}}(\mathbb{R}^n)$ and
$([L^{\vec{r}}(\mathbb{R}^n)]^{\frac{1}{p}})'$.
On the other hand, by \cite[Lemma~4.1]{gp1965},
we find that both $L^{\vec{r}}(\mathbb{R}^n)$
and $[L^{\vec{r}}(\mathbb{R}^n)]'$ have
absolutely continuous norms.
From these and Theorems~\ref{1931} and~\ref{1039},
we deduce the desired conclusions and hence
complete the proof of Theorem~\ref{2116}.
\end{proof}

\begin{remark}
\begin{enumerate}
\item[(i)]
To the best of our knowledge,
Theorem~\ref{2116} is completely new.
\item[(ii)]
Let $\vec{r}:=(r_1,\ldots,r_n)\in[1,\infty)^n$
with $r_-=1$. In this case,
since the Hardy--Littlewood maximal operator
may not be bounded on $[L^{\vec{r}}(\Omega)]'$
and since $[L^{\vec{r}}(\Omega)]'$
does not have an absolutely continuous norm,
it is still unclear whether or not
Theorem~\ref{2116} with $r_-=1$ holds true.
\end{enumerate}
\end{remark}

\subsection{Variable Lebesgue Spaces}\label{5.3}

Let $r\in\mathscr{M}(\mathbb{R}^n)$ be a nonnegative function
and let
\begin{equation*}
\widetilde{r}_-:=\underset{x\in\mathbb{R}^n}{
\mathop\mathrm{\,ess\,inf\,}}\,r(x)\ \text{and}\
\widetilde{r}_+:=\underset{x\in\mathbb{R}^n}{
\mathop\mathrm{\,ess\,sup\,}}\,r(x).
\end{equation*}
A nonnegative measurable function $r$
is said to be \emph{globally
log-H\"older continuous} if there exist
$r_{\infty}\in\mathbb{R}$ and a positive
constant $C$ such that, for any
$x,y\in\mathbb{R}^n$,
\begin{equation*}
|r(x)-r(y)|\le\frac{C}{\log(e+\frac{1}{|x-y|})}\ \ \text{and}\ \
|r(x)-r_\infty|\le \frac{C}{\log(e+|x|)}.
\end{equation*}
Recall that the \emph{variable Lebesgue space
$L^{r(\cdot)}(\mathbb{R}^n)$},
associated with a nonnegative measurable function
$r$, is defined to be the set
of all the $f\in\mathscr{M}(\mathbb{R}^n)$
having the following finite quasi-norm
\begin{equation*}
\|f\|_{L^{r(\cdot)}(\mathbb{R}^n)}:=\inf\left\{\lambda
\in(0,\infty):\ \int_{\mathbb{R}^n}\left[\frac{|f(x)|}
{\lambda}\right]^{r(x)}\,dx\le1\right\}.
\end{equation*}
The variable Lebesgue space
was introduced and studied
by Kov\'a$\check{\mathrm{c}}$ik and R\'akosn\'ik
\cite{kr1991}. In \cite{kr1991},
the variable Lebesgue space
was applied to study both the mapping properties
of Nemytskii operators and the related nonlinear elliptic
boundary value problems.
We refer the reader to
\cite{b2018,bbd2021,cf2013,cw2014,dhr2009,ns2012,n1950,n1951}
for more studies on variable Lebesgue spaces.
By the definition of $L^{r(\cdot)}(\mathbb{R}^n)$,
it is easy to prove that $L^{r(\cdot)}(\mathbb{R}^n)$
is a ball quasi-Banach function space
(see, for instance, \cite[Section~7.8]{shyy2017}).
When $1\leq\widetilde r_-\le \widetilde r_+<\infty$,
$(L^{r(\cdot)}(\mathbb{R}^n), \|\cdot\|_{
L^{r(\cdot)}(\mathbb{R}^n)})$ is
a Banach function space
in the sense of Bennett and Sharpley \cite{bs1988}
and hence also a ball Banach function space.
When $X:=L^{{r}(\cdot)}$,
we simply write
$\dot{W}^{1,r(\cdot)}(\Omega):=\dot{W}^{1,X}(\Omega)$.

Using Theorems~\ref{1931} and~\ref{1039},
we obtain the following conclusion.

\begin{theorem}\label{2041}
Let $\gamma\in\mathbb{R}\setminus\{0\}$
and $r:\ \mathbb{R}^n\to(0,\infty)$ be globally
log-H\"older continuous.
Let $1<\widetilde{r}_-\leq\widetilde{r}_+<\infty$
and $p\in[1,\widetilde{r}_-)$.
Then $f\in\dot{W}^{1,r(\cdot)}(\Omega)$ if and only if
$f\in L_{{\mathrm{loc}}}^1(\Omega)$ and
$$
\sup_{\lambda\in(0,\infty)}\lambda
\left\|\left[\int_{\Omega}
\mathbf{1}_{E_{\lambda,\frac{\gamma}{p}}[f]}(\cdot,y)
\left|\cdot-y\right|^{\gamma-n}\,dy\right]^\frac{1}{p}
\right\|_{L^{r(\cdot)}(\Omega)}<\infty,
$$
where $E_{\lambda,\frac{\gamma}{p}}[f]$
for any $\lambda\in(0,\infty)$
is the same as in \eqref{Elambda};
moreover, for such $f$,
$$
\sup_{\lambda\in(0,\infty)}\lambda
\left\|\left[\int_{\Omega}
\mathbf{1}_{E_{\lambda,\frac{\gamma}{p}}[f]}(\cdot,y)
\left|\cdot-y\right|^{\gamma-n}\,dy\right]^\frac{1}{p}
\right\|_{L^{r(\cdot)}(\Omega)}
\sim\left\|\,|\nabla f|\,\right\|_{L^{r(\cdot)}(\Omega)},
$$
where the positive equivalence constants are independent of $f$.
\end{theorem}

\begin{proof}
Notice that $[L^{r(\cdot)}(\mathbb{R}^n)]^{\frac{1}{p}}
=L^{\frac{r(\cdot)}{p}}(\mathbb{R}^n)$.
By this, $p\in[1,\widetilde{r}_-)$,
and \cite[Theorems~2.17 and~2.71]{cf2013},
we conclude that both
$L^{r(\cdot)}(\mathbb{R}^n)$ and
$L^{\frac{r(\cdot)}{p}}(\mathbb{R}^n)$
are ball Banach function spaces.
From \cite[Theorem~2.80]{cf2013} and
$1<\frac{\widetilde{r}_-}{p}\leq\widetilde{r}_+<\infty$,
we infer that
$([L^{r(\cdot)}(\mathbb{R}^n)]^{\frac{1}{p}})'
=L^{(\frac{r(\cdot)}{p})'}(\mathbb{R}^n)$,
where $(\frac{r(x)}{p})'$ is the conjugate index
of $\frac{r(x)}{p}$ for any $x\in\mathbb{R}^n$,
which, together with \cite[Theorem~1.7]{ahh2015},
further implies that the Hardy--Littlewood
maximal operator $\mathcal{M}$
is bounded on both
$L^{r(\cdot)}(\mathbb{R}^n)$ and
$([L^{r(\cdot)}(\mathbb{R}^n)]^{\frac{1}{p}})'$.
Moreover,
by \cite[p.\,73]{cf2013},
we conclude that both $L^{r(\cdot)}(\mathbb{R}^n)$
and $[L^{r(\cdot)}(\mathbb{R}^n)]'$
have absolutely continuous norms.
From these and Theorems~\ref{1931},
we deduce the desired conclusions and hence
complete the proof of Theorem~\ref{2041}.
\end{proof}

\begin{remark}
\begin{enumerate}
\item[(i)]
To the best of our knowledge,
Theorem~\ref{2041} is completely new.
\item[(ii)]
Let $r:\ \mathbb{R}^n\to(0,\infty)$ be globally
log-H\"older continuous with
$1=\widetilde{r}_-\leq\widetilde{r}_+<\infty$.
In this case, since $[L^{r(\cdot)}(\Omega)]'$
does not have an absolutely continuous
norm, it is still unclear whether or not
Theorem~\ref{2041} with $\widetilde{r}_-=1$
holds true.
\end{enumerate}
\end{remark}

\subsection{Lorentz Spaces}\label{5.7}

Recall that, for any $r,\tau\in(0,\infty)$,
the \emph{Lorentz space $L^{r,\tau}(\mathbb{R}^n)$}
is defined to be the set of all the
$f\in\mathscr{M}(\mathbb{R}^n)$ having the following finite quasi-norm
\begin{equation*}
\|f\|_{L^{r,\tau}(\mathbb{R}^n)}
:=\left\{\int_0^{\infty}
\left[t^{\frac{1}{r}}f^*(t)\right]^\tau
\frac{\,dt}{t}\right\}^{\frac{1}{\tau}},
\end{equation*}
where $f^*$ denotes the \emph{decreasing rearrangement of $f$},
defined by setting, for any $t\in[0,\infty)$,
\begin{equation*}
f^*(t):=\inf\left\{s\in(0,\infty):\ \left|
\left\{x\in\mathbb{R}^n:\ |f(x)|>s\right\}\right|\leq t\right\}
\end{equation*}
with the convention $\inf \emptyset = \infty$.
As a generalization of Lebesgue spaces,
Lorentz spaces were originally studied by Lorentz \cite{l1950,l1951},
which prove the intermediate spaces of Lebesgue spaces in the real interpolation method
(see, for instance, \cite{c1964}).
We refer the reader to \cite{CF1,CF2,Hunt,osttw2012,st2001}
for more studies on Lorentz spaces
and to \cite{lwyy2019,lyy2016,lyy2017,
lyy2018,zhy2020}
for more studies on Hardy--Lorentz spaces.
When $r,\tau\in(0,\infty)$,
$L^{r,\tau}(\mathbb{R}^n)$ is a quasi-Banach function space
in the sense of Bennett and Sharpley \cite{bs1988}
and hence a ball quasi-Banach function space
(see, for instance, \cite[Theorem 1.4.11]{g2014});
when $r,\tau\in(1,\infty)$,
$L^{r,\tau}(\mathbb{R}^n)$ is a Banach function space
and hence a ball Banach function space
(see, for instance, \cite[p.\,87]{shyy2017}
and \cite[p.\,74]{g2014}).

Using Theorems~\ref{1931} and~\ref{1039},
we obtain the following conclusion.

\begin{theorem}\label{2047}
Let $\gamma\in\mathbb{R}\setminus\{0\}$,
$r,\tau\in(1,\infty)$, and $p\in[1,\min\{r,\tau\})$.
Then $f\in\dot{W}^{1,L^{r,\tau}}(\Omega)$ if and only if
$f\in L_{{\mathrm{loc}}}^1(\Omega)$ and
$$
\sup_{\lambda\in(0,\infty)}\lambda
\left\|\left[\int_{\Omega}
\mathbf{1}_{E_{\lambda,\frac{\gamma}{p}}[f]}(\cdot,y)
\left|\cdot-y\right|^{\gamma-n}\,dy\right]^\frac{1}{p}
\right\|_{L^{r,\tau}(\Omega)}<\infty,
$$
where $E_{\lambda,\frac{\gamma}{p}}[f]$
for any $\lambda\in(0,\infty)$
is the same as in \eqref{Elambda};
moreover, for such $f$,
$$
\sup_{\lambda\in(0,\infty)}\lambda
\left\|\left[\int_{\Omega}
\mathbf{1}_{E_{\lambda,\frac{\gamma}{p}}[f]}(\cdot,y)
\left|\cdot-y\right|^{\gamma-n}\,dy\right]^\frac{1}{p}
\right\|_{L^{r,\tau}(\Omega)}
\sim\left\|\,|\nabla f|\,\right\|_{L^{r,\tau}(\Omega)},
$$
where the positive equivalence constants are independent of $f$.
\end{theorem}

\begin{proof}
Notice that
$[L^{r,\tau}(\mathbb{R}^n)]^\frac{1}{p}
=L^{\frac{r}{p},\frac{\tau}{p}}(\mathbb{R}^n)$.
By this, $p\in[1,\min\{r,\tau\})$,
and the conclusion in \cite[p.\,87]{shyy2017},
we conclude that both $L^{r,\tau}(\mathbb{R}^n)$
and $[L^{r,\tau}(\mathbb{R}^n)]^\frac{1}{p}$
are ball Banach function spaces.
From \cite[Theorem~1.4.16(vi)]{g2014}, we infer that
$([L^{r,\tau}(\mathbb{R}^n)]^\frac{1}{p})'
=L^{(\frac{r}{p})',(\frac{\tau}{p})'}(\mathbb{R}^n)$,
which, combined with \cite[Lemma~3.5]{zhy2020},
further implies that
the Hardy--Littlewood maximal operator $\mathcal{M}$
is bounded on both $L^{r,\tau}(\mathbb{R}^n)$ and
$([L^{r,\tau}(\mathbb{R}^n)]^\frac{1}{p})'$.
Moreover, by \cite[Remark~3.4(iii)]{wyy2020},
we find that both $L^{r,\tau}(\mathbb{R}^n)$
and $[L^{r,\tau}(\mathbb{R}^n)]'$ have
absolutely continuous norms.
From these and Theorems~\ref{1931} and~\ref{1039},
we deduce the desired conclusions and hence
complete the proof of Theorem~\ref{2047}.
\end{proof}

\begin{remark}
\begin{enumerate}
\item[\textup{(i)}]
To the best of our knowledge,
Theorem~\ref{2047} is completely new.
\item[\textup{(ii)}]
Let $r,\tau\in[1,\infty)$ with $r=1$ or $\tau=1$.
In this case, since $L^{r,\tau}(\Omega)$ might not be
a ball Banach function space and since
$[L^{r,\tau}(\Omega)]'$ does not have an absolutely
continuous norm, it is still unclear whether or not
Theorem~\ref{2047} in this case holds true.
\end{enumerate}
\end{remark}

\subsection{Orlicz Spaces}\label{5.5}

A non-decreasing function $\Phi:\ [0,\infty)
\ \to\ [0,\infty)$ is called an \emph{Orlicz function} if
\begin{enumerate}
\item[(i)]
$\Phi(0)= 0$;
\item[(ii)]
for any $t\in(0,\infty)$,
$\Phi(t)\in(0,\infty)$;
\item[(iii)]
$\lim_{t\to\infty}\Phi(t)=\infty$.
\end{enumerate}
An Orlicz
function $\Phi$ is said to be of \emph{lower}
(resp. \emph{upper}) \emph{type} $r$ for some
$r\in\mathbb{R}$ if there exists 
$C_{(r)}\in(0,\infty)$ such that,
for any $t\in[0,\infty)$ and
$s\in(0,1)$ [resp. $s\in[1,\infty)$],
$\Phi(st)\le C_{(r)} s^r\Phi(t)$.
Throughout this subsection, we always assume that
$\Phi:\ [0,\infty)\ \to\ [0,\infty)$
is an Orlicz function with both positive lower
type $r_{\Phi}^-$ and positive upper type
$r_{\Phi}^+$.
The \emph{Orlicz space $L^\Phi(\mathbb{R}^n)$}
is defined to be the set of all the
$f\in\mathscr{M}(\mathbb{R}^n)$
having the following finite quasi-norm
\begin{equation*}
\|f\|_{L^\Phi(\mathbb{R}^n)}:=\inf\left\{\lambda\in
(0,\infty):\ \int_{\mathbb{R}^n}\Phi\left(\frac{|f(x)|}
{\lambda}\right)\,dx\le1\right\}.
\end{equation*}
The Orlicz space was introduced by Birnbaum and Orlicz \cite{bo1931}
and Orlicz \cite{o}, which is widely used in
various branches of analysis.
We refer the reader to \cite{b2021,dfmn2021,ns2014,rr2002,ylk2017}
for more studies on Orlicz spaces.
It is easy to show that $L^\Phi(\mathbb{R}^n)$
is a quasi-Banach function space
in the sense of Bennett and Sharpley \cite{bs1988}
(see \cite[Section~7.6]{shyy2017})
and hence a ball quasi-Banach function space.
In particular,
when $1\leq r^-_\Phi\leq r^+_\Phi<\infty$,
$L^\Phi(\mathbb{R}^n)$ is a Banach function space
and hence a ball Banach function space.
When $X:=L^{\Phi}$,
we simply write
$\dot{W}^{1,\Phi}(\Omega):=\dot{W}^{1,X}(\Omega)$.

Using Theorems~\ref{1931} and~\ref{1039},
we obtain the following conclusion.

\begin{theorem}\label{2051}
Let $\gamma\in\mathbb{R}\setminus\{0\}$ and
$\Phi$ be an Orlicz function with both
lower type $r^-_{\Phi}$
and upper type $r^+_\Phi$.
Let $1<r^-_{\Phi}\leq r^+_{\Phi}<\infty$
and $p\in[1,r^-_{\Phi})$.
Then $f\in\dot{W}^{1,\Phi}(\Omega)$ if and only if
$f\in L_{{\mathrm{loc}}}^1(\Omega)$ and
$$
\sup_{\lambda\in(0,\infty)}\lambda
\left\|\left[\int_{\Omega}
\mathbf{1}_{E_{\lambda,\frac{\gamma}{p}}[f]}(\cdot,y)
\left|\cdot-y\right|^{\gamma-n}\,dy\right]^\frac{1}{p}
\right\|_{L^{\Phi}(\Omega)}<\infty,
$$
where $E_{\lambda,\frac{\gamma}{p}}[f]$
for any $\lambda\in(0,\infty)$
is the same as in \eqref{Elambda};
moreover, for such $f$,
$$
\sup_{\lambda\in(0,\infty)}\lambda
\left\|\left[\int_{\Omega}
\mathbf{1}_{E_{\lambda,\frac{\gamma}{p}}[f]}(\cdot,y)
\left|\cdot-y\right|^{\gamma-n}\,dy\right]^\frac{1}{p}
\right\|_{L^{\Phi}(\Omega)}
\sim\left\|\,|\nabla f|\,\right\|_{L^{\Phi}(\Omega)},
$$
where the positive equivalence constants are independent of $f$.
\end{theorem}

\begin{proof}
Let $\Phi_p(t):=\Phi(t^{\frac{1}{p}})$ for any $t\in[0,\infty)$.
Notice that
$[L^\Phi(\mathbb{R}^n)]^\frac{1}{p}=L^{\Phi_p}(\mathbb{R}^n)$.
From this and $p\in[1,r^-_{\Phi})$,
it is easy to prove that both $L^\Phi(\mathbb{R}^n)$
and $[L^\Phi(\mathbb{R}^n)]^\frac{1}{p}$
are ball Banach function spaces.
By \cite[Theorem~1.2.1]{kk1991} and \cite[Theorem~13]{rr2002},
we conclude that
$\mathcal{M}$ is bounded on both $L^\Phi(\mathbb{R}^n)$ and
$([L^\Phi(\mathbb{R}^n)]^{\frac{1}{p}})'$.
In addition, from the proof of \cite[Lemma~4.5]{zyyw2019},
we infer that both $L^\Phi(\mathbb{R}^n)$
and $[L^\Phi(\mathbb{R}^n)]'$
have absolutely continuous norms.
By these and Theorems~\ref{1931} and~\ref{1039},
we obtain the desired conclusions and hence
complete the proof Theorem~\ref{2051}.
\end{proof}

\begin{remark}
\begin{enumerate}
\item[\textup{(i)}]
To the best of our knowledge,
Theorem~\ref{2051} is completely new.
\item[\textup{(ii)}]
Let $\Phi$ be an Orlicz function with both
positive lower type $r^-_{\Phi}$
and positive upper type $r^+_\Phi$
such that $1=r^-_{\Phi}\leq r^+_{\Phi}<\infty$.
In this case, since $[L^\Phi(\Omega)]'$
does not have an absolutely continuous norm,
it is still
unclear whether or not Theorem~\ref{2051}
in this case holds true.
\end{enumerate}
\end{remark}

\subsection{Orlicz-Slice Spaces}\label{5.6}

Throughout this subsection, we always assume
that $\Phi: [0,\infty)\to [0,\infty)$
is an Orlicz function with both positive
lower type $r_{\Phi}^-$ and positive upper
type $r_{\Phi}^+$. For any given $t,r\in(0,\infty)$,
the \emph{Orlicz-slice space}
$(E_\Phi^r)_t(\mathbb{R}^n)$ is defined to be the set of all
the $f\in\mathscr{M}(\mathbb{R}^n)$ having the following finite quasi-norm
\begin{equation*}
\|f\|_{(E_\Phi^r)_t(\mathbb{R}^n)} :=\left\{\int_{\mathbb{R}^n}
\left[\frac{\|f\mathbf{1}_{B(x,t)}\|_{L^\Phi(\mathbb{R}^n)}}
{\|\mathbf{1}_{B(x,t)}\|_{L^\Phi(\mathbb{R}^n)}}\right]
^r\,dx\right\}^{\frac{1}{r}}.
\end{equation*}
These spaces were originally introduced in
\cite{zyyw2019} as a generalization of both
the slice space of Auscher and Mourgoglou
\cite{am2019,ap2017} and the Wiener amalgam space
in \cite{h2019,h1975,kntyy2007}.
We refer the reader to \cite{hkp2021,hkp2022,zyy2022,zyyw2019}
for more studies on Orlicz-slice spaces.
By \cite[Lemma 2.28]{zyyw2019} and \cite[Remark 7.41(i)]{zwyy2021},
we find that the Orlicz-slice
space $(E_\Phi^r)_t(\mathbb{R}^n)$ is a
ball Banach function space, but in general is not a
Banach function space
in the sense of Bennett and Sharpley \cite{bs1988}.

Using Theorems~\ref{1931} and~\ref{1039},
we obtain the following conclusion.

\begin{theorem}\label{2055}
Let $\gamma\in\mathbb{R}\setminus\{0\}$
and $\Phi$ be an Orlicz function with
both positive lower type $r^-_{\Phi}$
and positive upper type $r^+_\Phi$.
Let $t\in(0,\infty)$, $r\in(1,\infty)$,
$1<r^-_{\Phi}\leq r^+_{\Phi}<\infty$,
and $p\in[1,\min\{r^-_{\Phi},\,r\})$.
Then $f\in\dot{W}^{1,(E_\Phi^r)_t}(\Omega)$
if and only if
$f\in L_{{\mathrm{loc}}}^1(\Omega)$ and
$$
\sup_{\lambda\in(0,\infty)}\lambda
\left\|\left[\int_{\Omega}
\mathbf{1}_{E_{\lambda,\frac{\gamma}{p}}[f]}(\cdot,y)
\left|\cdot-y\right|^{\gamma-n}\,dy\right]^\frac{1}{p}
\right\|_{(E_\Phi^r)_t(\Omega)}<\infty,
$$
where $E_{\lambda,\frac{\gamma}{p}}[f]$
for any $\lambda\in(0,\infty)$
is the same as in \eqref{Elambda};
moreover, for such $f$,
$$
\sup_{\lambda\in(0,\infty)}\lambda
\left\|\left[\int_{\Omega}
\mathbf{1}_{E_{\lambda,\frac{\gamma}{p}}[f]}(\cdot,y)
\left|\cdot-y\right|^{\gamma-n}\,dy\right]^\frac{1}{p}
\right\|_{(E_\Phi^r)_t(\Omega)}
\sim\left\|\,|\nabla f|\,\right\|_{(E_\Phi^r)_t(\Omega)},
$$
where the positive equivalence constants are independent of $f$.
\end{theorem}

\begin{proof}
Notice that
$[(E_\Phi^r)_t(\mathbb{R}^n)]^\frac{1}{p}
=(E_{\Phi_p}^\frac{r}{p})_t(\mathbb{R}^n)$,
where $\Phi_p(t):=\Phi(t^{\frac{1}{p}})$
for any $t\in[0,\infty)$.
From this, $p\in[1,\min\{r^-_{\Phi},\,r\})$,
\cite[Lemma~2.28]{zyyw2019},
and \cite[Remark~7.41(i)]{zwyy2021},
we deduce that both $(E_\Phi^r)_t(\mathbb{R}^n)$
and $[(E_\Phi^r)_t(\mathbb{R}^n)]^\frac{1}{p}$
are ball Banach function spaces.
By \cite[Theorem~2.26 and Lemma~4.4]{zyyw2019},
we conclude that
the Hardy--Littlewood maximal operator $\mathcal{M}$ is
bounded on both $(E_\Phi^r)_t(\mathbb{R}^n)$ and
$([(E_\Phi^r)_t(\mathbb{R}^n)]^\frac{1}{p})'$.
From \cite[Lemma~4.5]{zyyw2019},
we infer that both $(E_\Phi^r)_t(\mathbb{R}^n)$
and $[(E_\Phi^r)_t(\mathbb{R}^n)]'$
have absolutely continuous norms.
Using these and Theorems~\ref{1931} and~\ref{1039},
we obtain the desired conclusions and hence
complete the proof of Theorem~\ref{2055}.
\end{proof}

\begin{remark}
\begin{enumerate}
\item[\textup{(i)}]
To the best of our knowledge,
Theorem~\ref{2055} is completely new.
\item[\textup{(ii)}]
Let $\Phi$ be an Orlicz function with both
positive lower type $r^-_{\Phi}$
and positive upper type $r^+_\Phi$
such that $1=r^-_{\Phi}\leq r^+_{\Phi}<\infty$.
In this case, since $[(E_\Phi^r)_t(\Omega)]'$
does not have an absolutely continuous norm,
it is still
unclear whether or not Theorem~\ref{2055}
in this case holds true.
\end{enumerate}
\end{remark}

\noindent\textbf{Acknowledgements}\quad
Chenfeng Zhu would like to thank Yangyang Zhang
and Yinqin Li for some useful discussions and suggestions
on the subject of this article. Dachun Yang would
like to thank Professor Emiel Lorist and Professor
Zoe Nieraeth for some helpful discussions
on Banach function spaces.

\bigskip

\noindent  Chenfeng Zhu, Dachun Yang (Corresponding author) and Wen Yuan

\medskip

\noindent Laboratory of Mathematics
and Complex Systems (Ministry of Education of China),
School of Mathematical Sciences,
Beijing Normal University,
Beijing 100875, The People's Republic of China

\smallskip

\noindent{\it E-mails:} \texttt{cfzhu@mail.bnu.edu.cn} (C. Zhu)

\noindent\phantom{{\it E-mails:} }\texttt{dcyang@bnu.edu.cn} (D. Yang)

\noindent\phantom{{\it E-mails:} }\texttt{wenyuan@bnu.edu.cn} (W. Yuan)

\end{document}